\documentclass[phd,tocprelim]{cornell}
\usepackage{graphicx, pstricks}
\usepackage{hangcaption}
\usepackage{palatino}

\usepackage{amsthm}  
\usepackage{amsmath}
\usepackage{epsfig}
\usepackage{graphicx}
\usepackage{amsfonts, amscd}
\usepackage{amssymb, stmaryrd}
\usepackage{color, xypic}
\usepackage{epsfig,latexsym}

\input xy
\xyoption{all}

\renewcommand{\caption}[1]{\singlespacing\hangcaption{#1}\normalspacing}

\newtheorem{Thm}{Theorem}

\newtheorem{Cor}{Corollary}
\newtheorem{Lem}{Lemma}
\newtheorem{Prop}{Proposition}

\newtheorem{Exa}{Example}

\newtheoremstyle{ex_contd}{}{}{\itshape}{}{\bfseries}{}{ }{\thmname{#1}\thmnote{ (\ndseries #3)}}
\theoremstyle{ex_contd}
\newtheorem{cont_ex}{Continuation of Example }  

\newcommand{\ra}{\rightarrow}
\newcommand{\C}{\mathbb{C}}
\newcommand{\Z}{\mathbb{Z}}
\newcommand{\Q}{\mathbb{Q}}
\newcommand{\R}{\mathbb{R}}
\newcommand{\N}{\mathbb{N}}

\newcommand{\af}{\text{\em AffFlag}_}
\newcommand{\ag}{\text{\em AffGr}_}
\newcommand\Tr{{\rm Tr\,}}
\newcommand\init{\mathop{\rm init}\space}
\newcommand\rank{\mathop{\rm rank}\space}

\def\Le{{\rotatebox[origin=c]{180}{$\Gamma$}}}

\font\co=lcircle10
\def\petit#1{{\scriptstyle #1}}

\def\jr{\smash{\raise2pt\hbox{\co \rlap{\rlap{\char'005} \char'007}}
               \raise6pt\hbox{\rlap{\vrule height6.5pt}}
               \raise2pt\hbox{\rlap{\hskip4pt \vrule height0.4pt depth0pt
                width7.7pt}}}}
\def\je{\smash{\raise2pt\hbox{\co \rlap{\rlap{\char'005}
                \phantom{\char'007}}}\raise6pt\hbox{\rlap{\vrule height6pt}}}}
\def\+{\smash{\lower2pt\hbox{\rlap{\vrule height14pt}}
                \raise2pt\hbox{\rlap{\hskip-3pt \vrule height.4pt depth0pt
                width14.7pt}}}}
\def\perm#1#2{\hbox{\rlap{$\petit {#1}_{\scriptscriptstyle #2}$}}%
                \phantom{\petit 1}}
\def\phperm{\phantom{\perm w3}}

\def\textcross{\ \smash{\lower4pt\hbox{\rlap{\hskip4.15pt\vrule height14pt}}
                \raise2.8pt\hbox{\rlap{\hskip-3pt \vrule height.4pt depth0pt
                width14.7pt}}}\hskip12.7pt}
\def\textelbow{\ \hskip.1pt\smash{\raise2.8pt%
                \hbox{\co \hskip 4.15pt\rlap{\rlap{\char'005} \char'007}
                \lower6.8pt\rlap{\vrule height3.5pt}
                \raise3.6pt\rlap{\vrule height3.5pt}}
                \raise2.8pt\hbox{%
                  \rlap{\hskip-7.15pt \vrule height.4pt depth0pt width3.5pt}%
                  \rlap{\hskip4.05pt \vrule height.4pt depth0pt width3.5pt}}}
                \hskip8.7pt}

\title {Affine Patches on Positroid Varieties and Affine Pipe Dreams}
\author {Michelle Bernadette Snider}
\conferraldate {January}{2011}
\degreefield {Ph.D.}
\copyrightholder{Michelle Bernadette Snider}
\copyrightyear{2011}

\begin{document}

\maketitle
\makecopyright

\begin{abstract}
The objects of interest in this thesis are positroid varieties in the Grassmannian, which are indexed by juggling patterns. In particular, we study affine patches on these positroid varieties. Our main result corresponds these affine patches to Kazhdan-Lusztig varieties in the affine Grassmannian. We develop a new term order and study how these spaces are related to subword complexes and Stanley-Reisner ideals.  We define an extension of pipe dreams to the affine case and conclude by showing how our affine pipe dreams are generalizations of Cauchon and \Le- diagrams.
\end{abstract}

\contentspage
\figurelistpage

\normalspacing \setcounter{page}{1} \pagenumbering{arabic}
\pagestyle{cornell} \addtolength{\parskip}{0.5\baselineskip}

\chapter{Introduction}

Given a space of matrices, we may impose various rank conditions which yield algebraic varieties with interesting geometric and combinatorial properties.
Matrix Schubert varieties are defined by putting rank conditions on the upper left submatrices of square matrices. These varieties can be indexed by permutations and have associated combinatorial diagrams called pipe dreams. In this thesis, we will be concerned with positroid varieties in the Grassmannian, defined by cyclic rank conditions, and indexed by juggling patterns rather than permutations. We generalize pipe dreams to this situation.

In particular, we will look at Schubert patches on positroid varieties, indexed by $\lambda=(\lambda_1,\ldots, \lambda_k) \in \binom{[n]}{k}$. We look at $k \times (n-k)$ matrices where we set the $\lambda_i^{th}$ column equal to the $i^{th}$ column of the identity matrix. Then we define a term order such that the initial ideal generated by the cyclic determinants is a product of all the variables. Using a juggling pattern, we put rank conditions on the resulting matrices. Our choice of term order allows us to apply results from \cite{Kfs} and \cite{Ksw} to show that the initial ideal is the Stanley-Reisner ideal of a subword complex for a particular word.  We define the affine analog of pipe dreams on an infinite strip. The main theorem of this thesis gives a geometric explanation for why the components of the initial ideals of these varieties give affine pipe dreams.

\chapter{Combinatorics Background}

We introduce the combinatorial objects that will be relevant to our main theorems, along with some interesting background and motivation. We start with permutations and affine permutations. We define juggling patterns, and correlate these to permutations. Finally we introduce diagrams called pipe dreams, to which we will return in the geometric sections. Our references for this section are \cite{Stanley} and \cite{Humphreys}, and information about heaps can be found in \cite{Vien} and \cite{Stem}.

\section{Permutations}\label{sec:perms}

A {\bf permutation} is an element of the symmetric group $$S_n=\{\pi: (1,\ldots,n) \rightarrow (1,\ldots,n) \},$$ also known as the Weyl group $A_{n-1}$. It is a Coxeter group generated by simple transpositions, $\Sigma = \{s_i=(i,i+1)$ for $i=1 \ldots n-1 \}$.  The relations are
\begin{enumerate}
\item $s_i^2 = 1$,
\item $s_i s_j =s_j s_i$ for $|i-j|>1$, and
\item $s_i s_{i+1} s_i = s_{i+1} s_i s_{i+1}$ (the braid relation).
\end{enumerate}
We also set $S_\infty = \bigcup_n S_n$, under the natural inclusion $S_n \hookrightarrow S_{n+1}$ as the stabilizer of $n+1$. We will use {\bf one-line notation} for our permutations, where we simply write a permutation $\pi \in S_n$ as the list $\pi(1) \; \pi(2)\ldots \pi(n)$. For example, $\pi = 4123$ represents the permutation $\pi(1)=4$, $\pi(2)=1$, $\pi(3)=2$, $\pi(4)=3$. The transposition $s_i$ switches places $i$ and $i+1$ when operating on the right ($\pi$ to $\pi s_i$), and switches elements $i$ and $i+1$ when operating on the left ($\pi$ to $s_i \pi) $.

A {\bf partial permutation matrix $\underline{\pi}$} is a matrix that has entries 1 and 0, with at most one nonzero entry in each row and column. We define the {\bf permutation matrix} associated to $\pi$ as the matrix that has a 1 in $(i, \pi(i))$ and 0's elsewhere. We define a {\bf (Rothe) diagram} as the boxes left in the $n \times n$ grid after we cross out all boxes south of and east of each 1 in the permutation matrix. See Figure \ref{fig:permdia}.

\begin{figure}[htbp]
\begin{center}
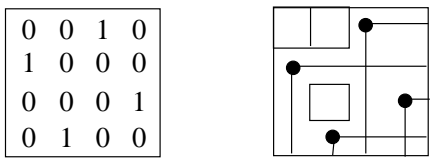
\caption{Matrix and diagram of the permutation 3142.}
\label{fig:permdia}
\end{center}
\end{figure}

Let $\pi^{-1}$ denote the {\bf inverse } of a permutation $\pi$, defined as the permutation that takes $\pi^{-1}(j)=i$ if and only if $\pi(i)=j$. Equivalently, $\pi \pi^{-1} = \pi^{-1} \pi = 1$, the identity permutation.

We have the following definitions, copied from \cite{KM}:
a \textbf{word} of size $m$ is an ordered sequence $Q=(\sigma_1,\ldots, \sigma_m)$ of elements of $\Sigma$. An ordered subsequence $P$ of $Q$ is called a \textbf{subword} of $Q$. Say $P$ \textbf{represents} $\pi \in S_n$ if the ordered product of the simple reflections in $P$ is a reduced decomposition for $\pi$. Say $P$ \textbf{contains} $\pi \in S_n$ if some subsequence of $P$ represents $\pi$.

An {\bf inversion} of $\pi \in S_n$ is a pair $(i,j)$ such that $i < j$ and $\pi(i) > \pi(j)$. The {\bf length $\ell(\pi)$} of $\pi$ is the number of inversions, and is so called since it is the length of the shortest word that represents $\pi$. This is also equal to the number of boxes in the diagram of $\pi$, since the number of boxes in row $j$ is given by $\# \{i \; | \; i<j, \; \pi(i)>\pi(j)\}$. We let $w_0$ denote the longest word, corresponding to $\pi=n \; n-1 \ldots 1)$. We write the rank of the $p \times q$ upper left submatrix of the any matrix, in particular a permutation matrix $\underline{\pi}$, as $r_{pq}(\underline{\pi}) = \# \{(i,j) \leq (p,q)\; |\; \pi(i) = j\}$, or just $r_{pq}$ if the permutation is clear from context.

We say that a word is {\bf 321-avoiding} if it has no decreasing subsequence of length 3. That is, if $w=w(1)\ldots w(m) \in S_m$, then there should not exist $1 \leq i < j<k\leq m$ such that $w(i) > w(j) > w(k)$. It has been shown in \cite{H} that a word in $S_n$ is 321-avoiding if and only if it has no reduced expression containing a substring of the form $s_i s_{i \pm 1} s_i$. We say a word $w$ is {\bf fully commutative} if one can get any reduced word from another by switching commuting generators (without braid relations). For $w \in S_n$, all reduced words for $v$ are related by just transpositions if and only if $v$ is 321-avoiding.

To any permutation, one can associate a poset (partially ordered set) called a {\bf heap}, whose vertices are labeled by simple transpositions (the letters of $w$), and such that the linear extensions of the heap encode all the reduced expressions for $w$. We can construct a heap by ``dropping" in the letters of $w$, where $i$ falls into column $i$ and if there is an $i-1$ or and $i+1$ in row $j$, then $i$ comes to a rest in row $j+1$ (where the base is row 1 and we count up). We define a {\bf wiring diagram} as a set of wires, one for each integer, where a transposition is represented by a $\times$, and the permutation is read left to right off the bottom, or bottom to top on the right, depending on the orientation of the diagram. It is straightforward to construct a wiring diagram from the heap of a permutation, as it corresponds to replacing each letter with a cross, and extending the ends of the wires north and south. See Figure \ref{fig:41523}.

\begin{figure}[htbp]
\begin{center}
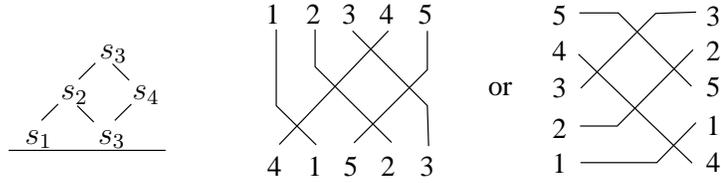
\caption{The heap and wiring diagram of $\pi=41523=s_3 s_1 s_4 s_2 s_3$.}
\label{fig:41523}
\end{center}
\end{figure}

For a general Coxeter group $W$, an element $w \in W$ is {\bf Grassmannian} if there is at most one $r_\alpha$ such that $w r_\alpha < w$ for $r_\alpha$ a simple reflection, or equivalently if it is a minimal length representative of the coset $S_n / (S_k \times S_{n-k}) $. It is {\bf bi-Grassmannian} if there is at most one $r_\alpha$ such that $w r_\alpha < w$ and at most one $r_\beta$ such that $ r_\beta w <w$ for simple reflections $r_\alpha$ and $r_\beta$. In particular, a permutation $\pi$ is Grassmannian if it has at most one descent, and bi-Grassmannian if both $\pi$ and $\pi^{-1}$ are Grassmannian. Note that the diagram of a bi-Grassmannian permutation (other than the identity) has exactly one rectangle, and this characterizes such permutations.

A {\bf partition} is a finite weakly decreasing sequence of positive integers. A given partition $\lambda=(\lambda_1,\ldots,\lambda_n)$, where $\lambda_i \geq \lambda_{i+1}$, can be represented by a {\bf Young diagram}, a collection of boxes arranged in left-justified rows where row $i$ has $\lambda_i$ boxes. We have a bijection between Grassmannian permutations of $[n]$ with a descent only after place $k$ and the set of Young diagrams $\{\lambda \subseteq (n-k)^k\}$, given as follows: rotate the Young diagram $45^\circ$ counterclockwise and draw the wiring diagram with $k$ wires going along the rows of $\lambda$ and $n-k$ wires going down the columns. Label both ends of the wires by $1,\ldots, n$ starting from the bottom. See Figure \ref{fig:youngdia}. Then this wiring diagram represents $w_\lambda$, where the wires connect index $i$ in the west with $w_\lambda (i)$ in the east. The resulting permutation is Grassmannian since it preserves the order on $1,\ldots,k$ and also on $k+1,\ldots, n$.

\begin{figure}[htbp]
\begin{center}
\scalebox{0.6}{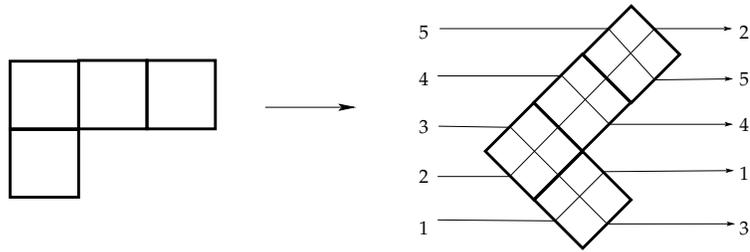}
\caption{Making the Grassmannian permutation $w_\lambda = 31452$ from Young diagram of shape $\lambda = (3,1)$.}
\label{fig:youngdia}
\end{center}
\end{figure}

We define a poset on permutations by the following partial orders. The {\bf weak Bruhat order} covering relations on $S_n$ are given by $w \succ v$ if $v=(i,i+1)w$ and $\ell(w) \leq \ell(v)$. The {\bf (strong) Bruhat order} covering relations are $w \succ v$, for $v= (ij)w$ where for $i<k<j$, $\pi(k)< \pi(i)$ or $\pi(k)>\pi(j)$ (``in between terms are not in between.")  See Figure \ref{fig:poset}. If $w< v$ in the Bruhat order, then $w>v$ in the {\bf opposite Bruhat order}.
\begin{figure}[htbp]
\begin{center}
\scalebox{0.7}{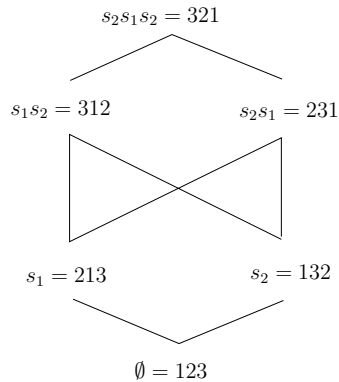}
\caption{Poset of $S_3$ to show the Bruhat order on permutations.}
\label{fig:poset}
\end{center}
\end{figure}

We say that an element in a finite poset $\mathcal{P}$ is {\bf basic} if it is not the unique greatest lower bound of the set $\{v \;|\; v>w; v,w \in \mathcal{P}\}$ (\cite{LS}). Since every non-basic element is then the unique greatest lower bound of those basic elements above, for some purposes, we need only determine properties of the basic elements, and some properties of the non-basic ones follow. We actually will consider a ``basic plus" set, that contains all the basic elements but may be larger.

\begin{Thm}\cite{LS} If $w \in W$ is basic, then it is bi-Grassmannian.\end{Thm}
\begin{proof} We prove the contrapositive: an element $w$ is not basic if and only if there exists $S \subseteq W$ where for all $s < w$, $s \in S$, the unique least upper bound of the elements in $S$ is $w$. An element $w$ is not bi-Grassmannian if there exist $r_{\alpha_1}$ and $r_{\alpha_2}$ such that $wr_{\alpha_1} < w $ and $wr_{\alpha_2} < w $. Say there exists $w' \in W$ such that $wr_{\alpha_1} < w' $ and $wr_{\alpha_2} < w' $. Then $w'= w r_{\alpha_1}r_{\alpha_i}=w r_{\alpha_2}r_{\alpha_j}$ for some $i,j$, implies that $r_{\alpha_2}r_{\alpha_1} = r_{\alpha_j}r_{\alpha_i}$, so $j=2$ and $i=1$. Then $w'=w$, the unique least upper bound of the elements $wr_{\alpha_1}$ and $wr_{\alpha_2}$, and thus not basic.
\end{proof}

An {\bf affine permutation} is an element of $$\hat{S_n}=\{\pi: \Z \rightarrow \Z  \mid \pi(i+n)=\pi(i)+n \; \forall \; i, \pi \textrm{ bijective}\}$$ In the case  that $\sum_i (\pi(i)-i) = 0$, we have the affine Weyl group $\widetilde{A_{n-1}}$. This is a Coxeter group (hence, with a Bruhat order) generated by $s_0,s_1,\ldots, s_{n-1}$, where
\begin{displaymath}
s_i(k) = \left\{ \begin{array}{ll}
k+1 & \textrm{if } k\equiv i \mod  n\\
k-1 & \textrm{if } k\equiv i+1 \mod n\\
k & \textrm{otherwise }
\end{array} \right.
\end{displaymath}
 Recall that the length of an element of a Coxeter group (affine or not) $\ell(\sigma)$ is the smallest integer $r$ such that we can write $\sigma$ as a product of $r$ simple reflections. In the $\widetilde{A_{n-1}}$ case, a formula for $\ell(\sigma)$ is given in \cite{Shi}:
\begin{displaymath} \ell(\sigma) = \sum_{1 \leq i < j \leq n} \left \lfloor \frac{|\pi(j)-\pi(i)|}{n} \right\rfloor
 \end{displaymath}
We also use one-line notation for affine permutations, where we just write one period $\pi(1) \ldots \pi(n)$, as the action of the permutation on any integer can be reconstructed using $\pi(i+n)=\pi(i)+n$.

For $\lambda \subseteq \binom{[n]}{k}$, let $bitstring(\lambda)=(c_1,\ldots, c_n)$ be the string where $$c(i)=\left\{ \begin{array}{ll}
1 & \textrm{if } i \in \lambda \\
0 & \textrm{if } i \notin \lambda
\end{array} \right.$$ We can associate to this a juggling pattern $f$ by $$(c_1,\ldots,c_n) \mapsto \{ f(i) = i+n c_{(i \mod n)} \}.$$
We have the following split exact sequence: 
$$ 1 \rightarrow \Z^n \rightarrow \hat{S_n} \rightarrow S_n \rightarrow 1,$$ so that $\hat{S_n} \cong S_n \ltimes \Z^n$. Also, $\hat{S_n} = \hat{A_{n-1}} \times <\{f(i)=i+1\}>$, and we use this to put a Bruhat order on each coset of $\hat{A_{n-1}}$.

\pagebreak

\section{Juggling Patterns}\label{sec:jp}

Our references for this section are \cite{KLS}, and in particular for juggling, \cite{Polster}.

We define the set of {\bf virtual juggling patterns} as the set of affine permutations, $$\{f: \Z \ra \Z \mid \forall \; i, f(i+n) = f(i) + n, f \textrm{ bijective}\}.$$ We can represent a juggling pattern as a {\bf siteswap}, a list of $n$ numbers representing the lengths of the throws $f(i)-i$. We need only list one cycle, with the understanding that the list of throws repeats both backwards and forwards in time. For example, the standard 3-ball cascade is represented by the siteswap $3$ for the list of throws $\ldots,3,3,3,\ldots$. Note that we allow a ball to travel either forwards or backwards in time, that is $f(i)-i$ is allowed to be negative. The latter case can be interpreted physically as an antimatter ball! It is standard lore among jugglers that
\begin{Thm} $k := \textrm{avg}(f(i)-i)$ is an integer, the number of balls (minus antiballs) in the pattern, \end{Thm}
so we can speak of a {\bf k-ball virtual juggling pattern}. If we restrict the virtual juggling patterns to those where $f(i) \geq i$, that is that we require that balls land after they are thrown, we refer to them simply as {\bf juggling patterns}. We add the condition that $\{ i \leq f(i) \leq i+n\}$ to get the finite set of {\bf bounded juggling patterns}. Note that neither forms a subgroup of $\hat{S_n}$.

One way to visually represent a juggling pattern is as a series of arcs connecting $i$ with $f(i)$ for all $i$. From this, it is easy to see that for each time $i$, one ball is caught and one ball is thrown. We call the special case when $f(i)=i$ a throw of length 0, or a {\it empty hand}. (A throw of length 2 is a {\it hold}, but we will not need this concept.)
\begin{figure}[htbp]
\begin{center}
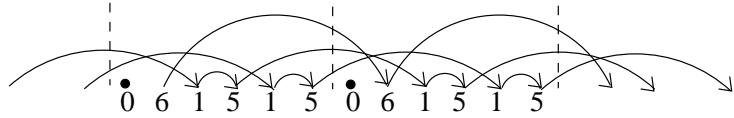
\caption{The siteswap $061515$.}
\label{fig:061515}
\end{center}
\end{figure}

The {\bf state} of a juggling pattern at time $i$ is the finite set $$\{ j \in \N \mid f^{-1}(i+j) \leq i \}.$$
At any given time, we can record the state of a juggling pattern as a list of $\times$'s representing the set of future times that the balls currently in the air will land, and $-$'s at times when no ball lands. For example $\times-\times - - \cdots$ means one ball is in the juggler's hand, and one ball is in the air that will land 2 counts from now; for simplicity we just write $\times-\times$. See Figure \ref{fig:24patterns} for all the possible states of $k=2$ balls with throws $f(i)-i \leq 4$. The arrows indicate what throws $f(i)-i$ can be made from each state. A bounded juggling pattern is a length $n$ cycle in this: for example, $4040$, $3022$, and $1304$. Note that if a state starts with a $-$, the only option is to wait one count for a ball to land, and if it starts with a $\times$, the next throw must land in an existing $-$.
\begin{figure}[htbp]
\begin{center}
\scalebox{0.8}{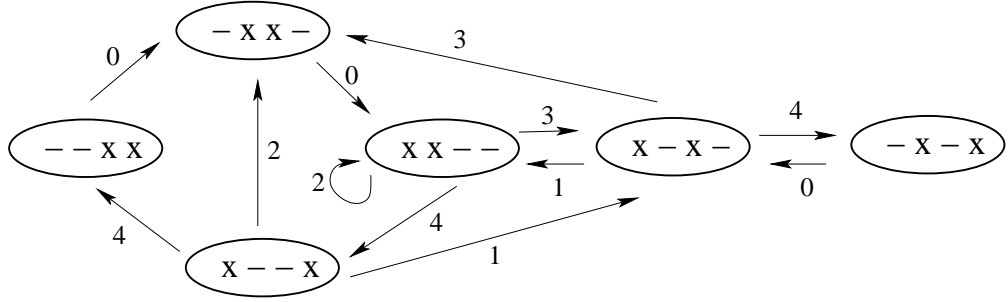}
\caption{State diagram for $n=4$ and $k=2$.}
\label{fig:24patterns}
\end{center}
\end{figure}

The {\bf ground state} for $k$ balls and length $n$ is $\times \times \cdots \times - - \cdots -$, with $k \; \times $'s and $(n-k) \; -$'s. We say that a juggling pattern is a {\bf ground state pattern} if its initial state is the ground state, or equivalently if one can add $kk\ldots k$ at the beginning. For example, the siteswap 566151 is ground state, but 661515 is not, since in 4444661515, 2 balls land at time 10 and it is therefore not a valid siteswap. We can construct a new pattern $f'$ from an existing pattern $f$ by taking a pair $(i,j)$, where $i<j$ and $f(i) > j$, and making $i$ a throw of length $j-i$, and $i+f(i)$ a throw of length $j+f(j)-i-f(i)$. Physically, this is equivalent to swapping the spot where ball $i$ lands with the spot where $j$ lands.

\begin{Lem}\cite[The Average Theorem \S 2.4 ]{Polster} The number of balls in the (non-virtual) juggling pattern $\pi$ is given by the formula  $k = \#\{ i \; | \; \pi(i) < i\}$, for any $i$.
\end{Lem}

For a permutation to be ground state, no ball thrown in the first $k$ spots can land in the first $k$ spots, as this would cause two balls to land at the same time when $k, \ldots,k$ is concatenated at the beginning. We can characterize those permutations corresponding to ground state juggling patterns by the following conditions: $k+1$ has to be in the first $k$ spots, $k+2$ in first $k+1$ spots, and so on. That is, in the first $k$ spots, all throws are greater than $ k$. Note that the ground state requires that if a fixed point of the associated finite permutation $\pi_f(i)=f(i)-i$ is in first $(n-k)$- spots, it is an $n$-throw, while if it occurs in the last $k$ spots, it is a 0-throw.

In the permutation $\pi$, consider $i,j$ satisfying $i<j$ and $\pi(j) < \pi(i)$, and having no arcs that start between $i$ and $j$ and end between $\pi(j)$ and $\pi(i)$. Then the Bruhat order on juggling patterns corresponds to: $\pi>\pi'$  if $\pi'(i)=\pi(j)$ and $\pi'(j)=\pi(i)$, where the covering relations require that there are no arcs starting between $i$ and $j$ and ending between $\pi(i)$ and $\pi(j)$. Graphically, this is shown in Figure \ref{fig:jugglingBruhat}. From this we see that in the case of a bounded juggling pattern $f$, the length of the associated affine permutation $\pi_f$ is the number of pairs of nested arcs.

\begin{figure}[htbp]
\begin{center}
\scalebox{0.4}{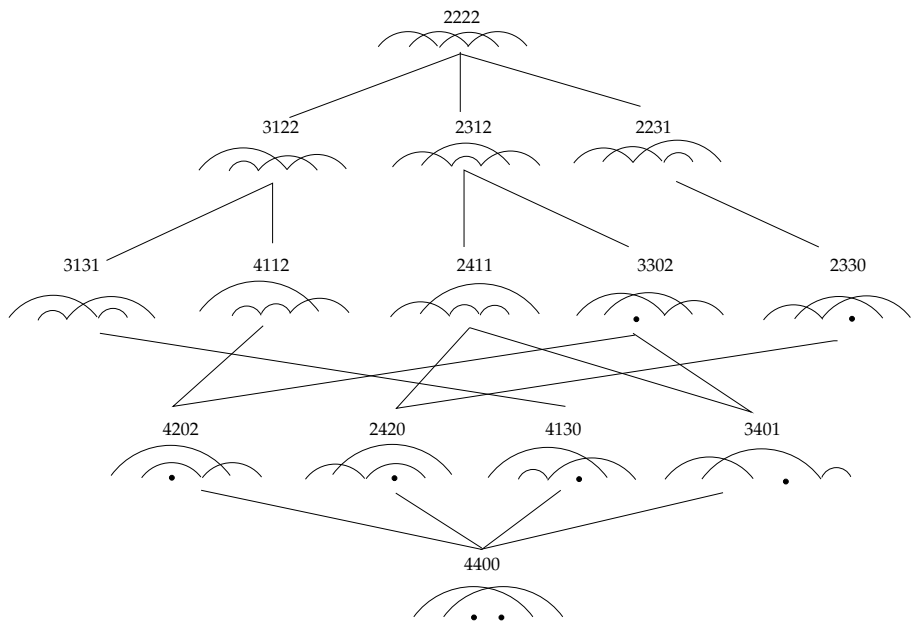}
\caption{The Bruhat order on bounded juggling patterns for $k=2$ and $n=4$.}
\label{fig:jugglingBruhat}
\end{center}
\end{figure}

\section{Pipe Dreams}\label{sec:pipedreams}

Our references for this section are \cite{BB} and \cite{MS}.

 A {\bf pipe dream} in $S_n$ is a diagram in an $n \times n$ square where each box is one of two tiles, elbows $\textelbow$ and crosses $\textcross$, such that all crosses occur above the southwest-northeast diagonal. Then we can think of the tiled grid as a set of pipes that begin on the north and east edges, and end on the west and south edges of the square, where the east to south pipes are always the same. We say a pipe dream is {\bf reduced} if no two pipes cross more than once. In this thesis, we will primarily be concerned with reduced pipe dreams, and will abuse definitions and use {\it pipe dream} to mean {\it reduced pipe dream} unless otherwise specified. Associated to each pipe dream is a permutation, which can be read off the diagram as follows: label the edges across the north side with $1,\ldots, n$, and label the edges down the west side with the same. Then follow each pipe from the north edge to the west edge, and label the end of the pipe with the same number. Then reading down the west side (from north to south) gives the associated permutation. Note that all the tiles in the lower triangle are elbows, so for simplicity we don't draw them. Pipe dreams were developed in \cite{BB}, under the name {\it RC-graphs}, to compute Schubert polynomials. Figure \ref{fig:2143pd} shows all the pipe dreams for the permutation $\pi=2143$.

 \begin{figure}[htbp]
\begin{center}
\scalebox{0.5}{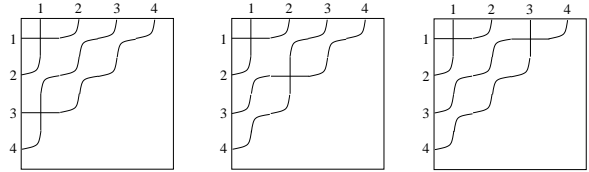}
\caption{All the (reduced) pipe dreams for the permutation $\pi=2143$.}
\label{fig:2143pd}
\end{center}
\end{figure}

For a given $\pi$, there is at least one pipe dream that gives the permutation, and there may be many. We use $\mathcal{RP}(\pi)$ to denote the set of all reduced pipe dreams of $\pi$. We consider two operations on a pipe dream that preserve the permutation: as in \cite{BB}, for $P$ a pipe dream, a {\bf ladder move} $\mathcal{L}_{ij}$ produces the diagram $P \bigcup {(i-m,j+1)} \backslash {(i,j)}$, as in Figure \ref{fig:chutesladders}. Note that for a ladder move, the two columns are adjacent but the number of rows is arbitrary. A {\bf chute move} $\mathcal{C}_{ij}$ is the transpose of a ladder move. We call the inverses of these moves the {\bf inverse chute} and {\bf inverse ladder} moves. We let $\mathcal{C}(D)$ be the set of pipe dreams that can be obtained from $D$ by a sequence of chute moves, and $\mathcal{L}(D)$ the same for ladder moves.

\begin{figure}[htbp]
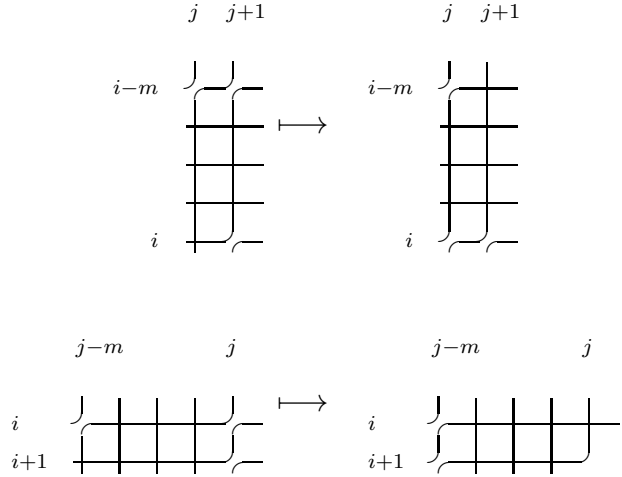

\begin{center}
\begin{eqnarray*} {\begin{array}{lccc}
 &\phperm & \perm{j}{} & \perm{j+1}{}\\
\phperm&\phperm &\phperm \\
\perm{i-m}{} && \jr & \jr \\
&\phperm& \+ & \+ \\
&\phperm& \+ & \+ \\
&\phperm& \+ & \+ \\
&\perm{i }{} & \+ & \jr \\
     \end{array}} & \longmapsto &
{\begin{array}{lccc}
 &\phperm& \perm{j}{} & \perm{j+1 }{}\\
&&\phperm &\phperm \\
\perm{i-m \; }{} && \jr & \+ \\
 && \+ & \+ \\
& & \+ & \+ \\
 && \+ & \+ \\
& \perm{i}{} & \jr & \jr \\
 \end{array}} \\
&&\\
&&\\
    {\begin{array}{rcccccc}
    &&\perm{j-m}{}  & & & & \perm{j}{}\\
    &&\phperm &\phperm &\phperm &\phperm &\phperm \\
    \perm{i}{}&&\jr & \+ & \+ & \+ & \jr \\
    \perm{i+1}{}&  & \+ & \+ & \+ & \+ & \jr \\  \end{array} }
    &\longmapsto &
    {\begin{array}{rcccccc}
    &&\perm{j-m }{} & & & & \perm{j}{} \\
    &&\phperm &\phperm &\phperm &\phperm &\phperm\\
    \perm{i}{}&&\jr & \+ & \+ & \+ & \+ \\
    \perm{i+1}{} & & \jr & \+ & \+ & \+ & \je \\ \end{array}}
\end{eqnarray*}
\caption{Ladder (above) and chute (below) moves.}\label{fig:chutesladders}
\end{center}
\end{figure}

\begin{Lem}\cite[Lemma 3.5]{BB} Ladder and chute moves preserve the permutation associated with a pipe dream. \end{Lem}

 We will make use of two distinguished pipe dreams: the {\bf bottom pipe dream}  $$D_{bot} (w) := \{(i,c) \; | \;c \leq m_{i} \}$$ where $m_i = \{j\; | \; j >i \text{ and } w_j < w_i \}$. Graphically, this pipe dream can be found by taking the permutation diagram, shoving all the blocks west, then replacing them with crosses and filling the rest of the diagram with elbows. Similarly, we have the {\bf top pipe dream}, where $D_{top} (w) = D_{bot}^t (w^{-1})$, where $t$ denotes transpose (\cite{BB}). Graphically, this corresponds to taking the permutation diagram, shoving all the blocks north and replacing them with crosses.
\pagebreak
 \begin{Thm}\cite[Theorem 3.7]{BB} Let $w \in S_\infty$. Then,
\begin{enumerate}
\item $D_{top}(w)$ does not admit an inverse chute.
\item Any element of $\mathcal{RP}(w)$ other than $D_{top}(w)$ admits an inverse chute.
\item $\mathcal{C}(D_{top}(w)) = \mathcal{RP}(w) = \mathcal{L}(D_{bot}(w))$.
\end{enumerate}
\end{Thm}

That is, every reduced pipe dream for $\pi$ can be obtained by chute and ladder moves on the bottom pipe dream, or equivalently by reverse chute and reverse ladder moves on the top pipe dream. See Figure \ref{fig:topbottom}.

\begin{figure}[htbp]
\begin{center}
\scalebox{0.7}{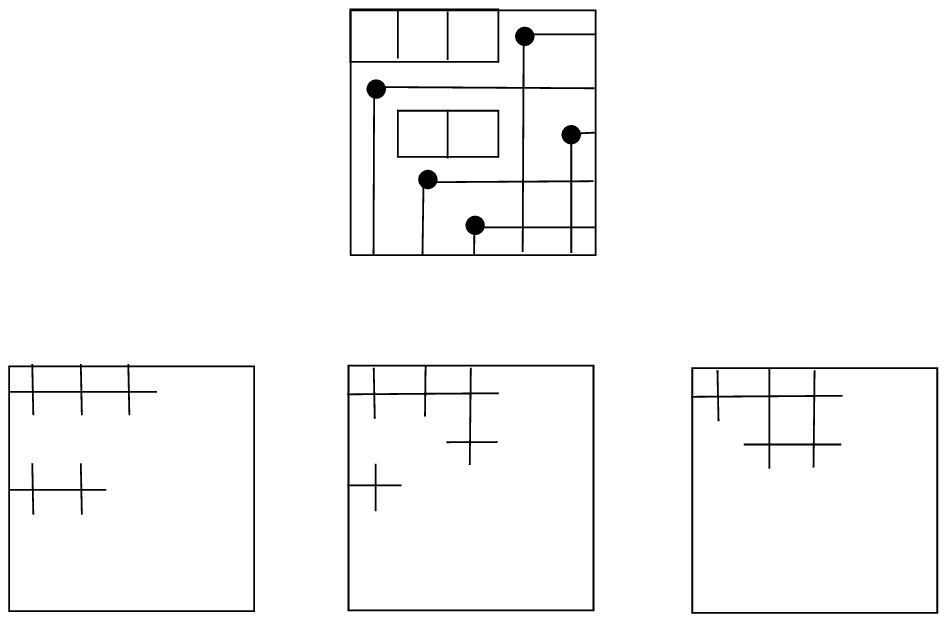}
\caption{For $\pi=41523$, the leftmost pipe dream is $D_{bot}$ and the rightmost is $D_{top}$. We have omitted elbows for clarity.}\label{fig:topbottom}
\end{center}
\end{figure}

An {\bf antidiagonal} is a subset $A \subseteq [n] \times [n]$ such that no element is (weakly) southeast of another: $(i,j) \in A$ and $(i,j) \leq (p,q) \Rightarrow (p,q) \notin A$. Consider the union over all $1 \leq p,q \leq n$ of the set of antidiagonals in $[p] \times [q]$ of size $1+r_{pq}(w)$. Then we define $\mathcal{A}_w$ to be the set of minimal elements under inclusion of this union. Recall that $\mathcal{RP}_w$ is the set of all reduced pipe dreams for $w$. Given a collection $\mathcal{C}$ of subsets of $[n] \times [n]$, a {\bf transversal} to $\mathcal{C}$ is a subset of $[n] \times [n]$ that meets every element of $C$ at least once. The {\bf transversal dual} of $\mathcal{C}$ is the set $\mathcal{C}^v$ of all minimal transversals to $\mathcal{C}$.

\begin{Thm}\cite{KM} and \cite[Theorem 3]{JM} For any permutation $\pi$, the transversal dual of the set $\mathcal{RP}_\pi$ of reduced pipe dreams for $\pi$ is the set $\mathcal{A}_\pi$ of antidiagonals for $\pi$. \end{Thm}

\begin{Exa} Let $\pi=1342$. The only essential box is at $(3,2)$, giving the condition $rank(M_{3,2} \leq 1)$. The antidiagonal set is $$\mathcal{A}_w = \{((2,1),(1,2)),((3,1),(1,2)),((3,1),(2,2)) \}$$ corresponding to three pipe dreams. See Figure \ref{fig:antidiags}.
\end{Exa}
\begin{figure}[htbp]
\begin{center}
\scalebox{1}{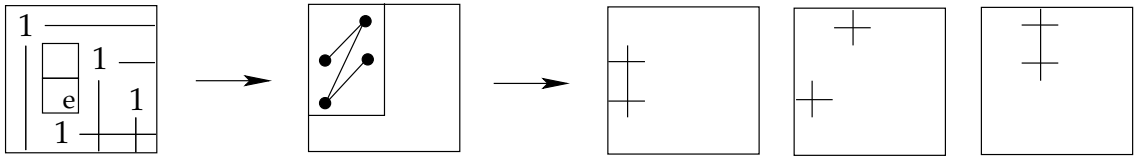}
\caption{For the single essential rank condition, the construction of pipe dreams from the transversal dual of the antidiagonals for $\pi=1342$.}
\label{fig:antidiags}
\end{center}
\end{figure}

\section{Simplicial Complexes }
We define a {\bf simplicial complex $\Delta$} on a set of ``vertices" $V$ as a downward order ideal in the power set of $V$. That is, the following condition holds: for $\sigma \in \Delta$, $\sigma' \subset \sigma$ implies that $ \sigma' \in \Delta$. We say a vertex $v$ is a {\bf cone vertex} if it lies in every maximal $\delta \in \Delta$. For simplicity we will often omit cone vertices from our diagrams (as they can be trivially re-added).
We call an element of $\Delta$ a {\bf face}, and call a maximal face a {\bf facet}. If all facets have the same size, which we assume hereafter, we say that $\Delta$ is {\bf pure}, and define a {\bf ridge} to be a face of one size lower. We say a face is {\bf exterior} if there exists a ridge $R \supseteq F$ where $R$ is itself contained in only one facet. If a face is not exterior, we say the face is {\bf interior}. A complex is {\bf thin} if each ridge is in only one or two facets, but not more.

For $\sigma$ a face in $\Delta$, the \textbf{deletion} of $\sigma$ from $\Delta$ is $\textrm{del}(\sigma,\Delta)=\{\sigma' \in \Delta | \sigma' \cap \sigma = \varnothing\}$. The \textbf{link} of $F$ in $\Delta$ is $\textrm{link}(\sigma,\Delta)=\{\sigma' \in \Delta | \sigma' \cap \sigma = \varnothing \textrm{ and } \sigma' \bigcup \sigma \in \Delta \}$. We say that $\Delta$ is \textbf{vertex-decomposable} if $\Delta$ is pure and either (1) $\Delta = \varnothing$, or (2) for some vertex $v \in \Delta$, both $\textrm{del}(v, \Delta)$ and $\textrm{link}(v, \Delta)$ are vertex-decomposable. A \textbf{shelling} of $\Delta$ is an ordered list $F_1,F_2,\ldots F_t$ of its facets such that $\bigcup_{j<i} \hat{F_j} \cap \hat{F_i}$ is a subcomplex generated by codimension 1 faces of $F_i$ for each $i \leq t$, where $\hat{F}$ denotes the set of faces of $F$. We say that $\Delta$ is \textbf{shellable} if it is pure and has a shelling. Then that intersection, $F_i \bigcap (F_1 \bigcup \cdots \bigcup F_{i-1})$ is isomorphic to a ball or a sphere.

\begin{Thm}\cite[Theorem 2.8]{Billera} If $\Delta$ is vertex-decomposable, then it is shellable. \end{Thm}

\begin{Thm}\cite[Proposition 1.2]{DK} A thin and shellable complex is homeomorphic to a ball. \end{Thm}

Define the {\bf pipe dream complex} $\Delta(\pi)$ to be the simplicial complex with vertices given by entries $(i,j)$ in $M_n\C$ and facets given by the elbow sets in pipe dreams for $\pi$. Then the lower-dimensional faces may be labeled with non-reduced pipe dreams. We will come back to this in \S \ref{sec:swc} in relation to subword complexes.

\begin{Lem}\cite{KM} $\Delta(\pi)$ is thin. \end{Lem}  
\begin{proof}
Lemma 3.5 in \cite{KM} says that if $w$ is a word in $\Pi$, and $\pi \in \Pi$ such that $|w| = \ell(\pi)+1$, then there are at most two elements $\nu \in T$ such that $w\ \nu$ represents $\pi$. Then for $R$ a ridge in $\Delta(\pi)$, $R$ is contained in 2 or 1 facet, depending on whether $R$ is a nonreduced pipe dream for $\pi$ or a reduced pipe dream for a permutation $\pi'> \pi$.  \end{proof}

We consider another interpretation of pipe dreams, related to rank conditions on matrices, and the ideals defined by the leading terms of the corresponding determinants. Let us consider a $k \times n$ matrix $\mathbf{x}=(x_{ij})$. We define an {\bf antidiagonal monomial} of size $r$ in $k[\mathbf{x}]$ as the product of the antidiagonal entries of an $r \times r $ submatrix of $\mathbf{x}$. Then for a $k \times n$ partial permutation $w$, the {\bf antidiagonal ideal} $J_w \subset k[\mathbf{x}]$ is generated by all antidiagonals in $\mathbf{x}_{p \times q}$ of size $1+r_{pq}(w)$ for all $p$ and $q$.

\chapter{Geometric Background}

We introduce the geometric objects that will be relevant to our main theorems, along with some interesting background and motivation. We begin with the definitions of some varieties and a combinatorial way to study them, simplicial complexes. We then consider a convenient term order that leads us to subword complexes, Stanley-Reisner rings, and Gr{\"o}bner bases.

\section{Subword Complexes, Stanley-Reisner Rings and Gr{\"o}bner Bases } \label{sec:swc}
The references for this section are \cite{KM} and Chapter 16 in \cite{MS}.

Let $Q$ be a word in a Coxeter group, and $\pi$ be a permutation. The \textbf{subword complex} $\Delta(Q,\pi)$ is simplicial complex whose faces are the set of subwords $Q\backslash P$ whose complements $P$ contain $\pi$. That is, if $Q \backslash D$ is a facet of the subword complex $\Delta(Q,\pi)$, then the reflections in $D$ give a reduced expression for $\pi$. See Figure \ref{fig:swgb}.

\begin{Lem}\cite[Lemma 2.2]{KM} $\Delta(Q,\pi)$ is a pure simplicial complex whose facets are the subwords $Q \backslash P$ such that $P \subseteq Q$ represents $\pi$. \end{Lem}

\begin{figure}[htbp]
\begin{center}
\scalebox{0.6}{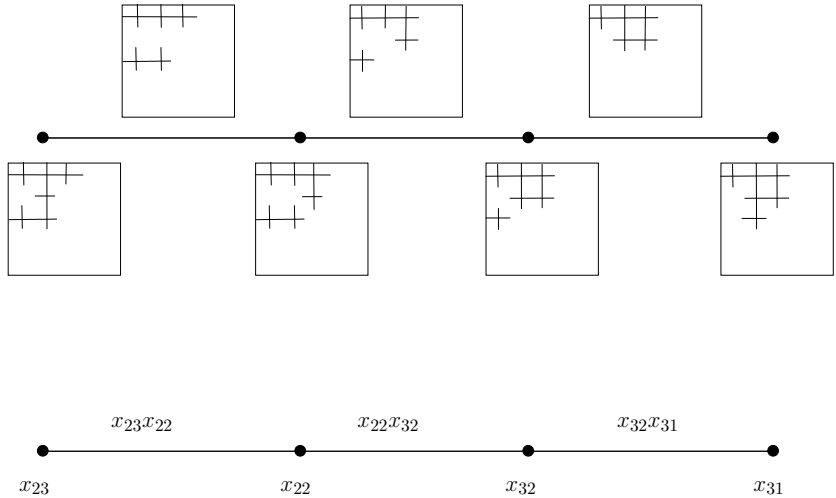}
\caption{The subword complex $\Delta(4321432434,\pi=41523=s_3 s_4 s_2 s_3 s_2)$, with components labeled by the pipe dream representation of the subword (above) and by the corresponding variables (below), both with cone points removed.}
\label{fig:swgb}
\end{center}
\end{figure}

\begin{Thm}\cite[Theorem 2.5]{KM} Subword complexes $\Delta(Q,\pi)$ are vertex-decomposable, hence shellable. \end{Thm}
The proof is by showing that both the link and deletion of the first letter in $Q$ are themselves subword complexes.

We will also need the fact that

\begin{Thm}\cite{Kfs}\label{thm:321} If $Q=Q'$ up to switching commuting letters, then $\Delta(Q,\pi)\cong \Delta(Q',\pi)$ for any $\pi$.  \end{Thm}

The following theorem shows that subword complexes are well-behaved.
\begin{Thm}\cite[Theorem 3.7]{KM} The subword complex $\Delta(Q,\pi)$ is homeomorphic to a ball or sphere; in particular, every ridge (codimension 1 facet) is contained in one or two facets. \end{Thm}

Fix a field $k$. A {\bf monomial ideal} in the polynomial ring $k[x_1, \ldots ,x_n]$ is an ideal that is generated by monomials. The {\bf Stanley-Reisner ring} of a simplicial complex $\Delta$ is the quotient ring $k[x_1,\ldots, x_n]/I_\Delta$, where $n = |V|$ and we define the monomial ideal $I_\Delta=\langle \prod_{j \in G} x_j \; | \; G \notin \Delta \rangle$. It is enough to take minimal such $G$.

\begin{Exa}
In Figure \ref{fig:swgb} are two copies of the subword complex for $\pi=41523$, one labeled with pipe dreams and the other with the corresponding variables. The nonfaces give the Stanley-Reisner ideal $$I_\Delta = \{x_{23}x_{31}, x_{23}x_{22}x_{32}, x_{23}x_{32}x_{31}, x_{23}x_{22}x_{31}, x_{22}x_{32}x_{31}, x_{23}x_{22}x_{32}x_{31}\}$$

\end{Exa}

When considering a polynomial ring over a field, it will be useful to put an ordering on the monomials. An order is said to be {\bf graded} if monomials are first ordered by decreasing total degree. The {\bf lexicographic order} compares two monomials of the same degree by highest power of the alphabetically first variable. If the powers of the first variable are equal, we compare the second variable's powers, and so on.

\begin{Exa} In $k[x,y,z]$, the graded lexicographic order gives  $$x^2 > xy > xz > y^2  > yz > z^2 > x > y > z>1 $$ \end{Exa}

The {\bf reverse lexicographic order} instead considers the powers of the last variable, and throws out the term with the highest power of the last variable. We repeat until terms with the last variable are gone, and repeat with the second to last variable. This gives the initial term, then we repeat the process to order the remaining terms. We will apply this in \S \ref{sec:newtermorder}. It is shown in Theorem 5 of \cite{Kfs} that the choice of term order does not matter; we will always end up with a single monomial.

\begin{Exa}(\S 3 in \cite{BB}) The bottom pipe dream for a permutation $\pi$ corresponds to the largest reduced word for $\pi$ in reverse lexicographic order. \end{Exa}

Given an ideal $I$ in our polynomial ring $I=\langle p_1, \ldots, p_k \rangle$, and for a fixed choice of ordering, the {\bf initial ideal $\init I$} is the ideal generated by all of the leading monomials in $I$. We say $G$ is a {\bf Gr{\"o}bner basis} of $I$ if the ideal given by the leading terms of polynomials in $I$ is already generated by the leading terms of the basis $G$, or equivalently that the leading term of any polynomial in $I$ is divisible by the leading term of some polynomial in $G$. In fact, (finite) Gr{\"o}bner bases always exist, and can be calculated for any ideal given a generating subset. The choice of ordering affects the number of calculations required, and reverse lexicographic ordering is typically the fastest (although we will not be concerned with this fact).

\section{Varieties}
A {\bf variety} in affine space is the set of solutions of a system of polynomial equations generating a prime ideal. In the following sections we will describe the equations we are considering, from rank conditions on matrices of certain dimensions and with certain specified columns.

We let $M_{k \times n}$ be the set of matrices over $\C$ (unless otherwise specified), with $k$ rows and $n$ columns, and typical element $M$. We will use $M_{[i,j]}$ to denote the submatrix composed of columns $i$ to $j$ of $M$. The {\bf general linear group $GL_k$} is the set of $k \times k$ invertible matrices. We say that a matrix $M \in M_{k \times n}(\R)$ is {\bf totally nonnegative} if the determinants of all of its $k \times k$ minors are nonnegative, and we denote the set of such matrices as $ M_{k \times n}^{\geq 0}(\R)$.

We consider varieties inside several different spaces. For a finite dimensional vector space $V$ over a field $k$, a {\bf partial flag} is a sequence of subspaces $$F=\{\emptyset = V_0 \subset V_1 \subset \cdots\subset V_k = V\}$$ where we let $d_i = \dim (V_i)$. If $k=n$ and $d_i = i$, then $F$ is a {\bf (complete) flag}. The set of all such flags forms the {\bf flag manifold}. We will also consider varieties that live inside the {\bf Grassmannian $Gr_k \C^n$}$=(GL_k \setminus M^{\text{rank=} k}_{k \times n})$, where $GL_k$ acting on the left does row operations. We can embed the Grassmannian as a particular subset of $\mathbb{P}(Alt^k \C^n)$, cut out by the Pl\"{u}cker equations, a fact we neither prove nor use.

A {\bf stratification} is a decomposition of a space into finitely many disjoint locally closed sets called {\bf strata}, such that every stratum's closure is a union of strata. Note that any finite decomposition of a space $X$ into disjoint locally closed sets can be refined to a stratification.

\subsection{Schubert Varieties}\label{sec:sch}

Our references for this section are \cite{Fultontxt} and \cite{Briontxt}.

Let $G$ be a connected reductive algebraic group over $\C$. Let $B$ be a Borel subgroup of $G$, $P$ be a parabolic subgroup of $G$, $N_{-}$ a maximal nilpotent group opposite B, and $T$ be the torus. Those not familiar with Lie theory can simply think of the case where $G=GL_n$, $B$ is the set of upper triangular matrices, $B_{-}$ denote the set of lower triangular matrices, $P \supseteq B$ is those matrices of the form $\{M \; | \; m_{ij} =0, \; i>k \geq j\}$ ($2 \times 2$-block upper triangular), $N_{-}$ is the set of lower triangular matrices with 1's on the diagonal, and $T$ is the maximal torus, $T = (\C^x)^n$. We identify $(G/P)^T$ with $W / W_p$ by $W W_p \mapsto WP/P$ for $W$ a Weyl group.
 \pagebreak
For $\pi \in W$ a Weyl group element, let $X_\pi^\circ := B_{-} \pi P/P \subseteq G/P$ and $X_\pi := \overline{X_\pi^\circ}$ be the associated {\bf Bruhat cell} and {\bf Schubert variety} respectively, each of codimension $\ell(\pi)$ (the length of $\pi$ as an element of the Coxeter group $W$). Define the {\bf opposite Bruhat cell} $X_\circ ^v := BvB/P$ and {\bf opposite Schubert variety} $X^v := \overline{X_\circ ^v}$, each of dimension $\ell(v)$. Each Bruhat cell or opposite Bruhat cell is just a copy of affine space. A {\bf Richardson variety} $X_u^w $ is the intersection of a Schubert variety with an opposite Schubert variety, $X_u^w = X_u \cap X^w.$

We call $X_1^\circ$, which is open and dense in $G/P$, the {\bf big cell}. We can shift the big cell to be ``centered" at $v$ by $v N_{-}B$, and call it the {\bf permuted big cell.} We define the {\bf Schubert patch} on $X_w$ as the intersection of $X_w$ with the permuted big cell, $X_w|_v=  X_w \cap (v N_{-}BvB)$. We note that the set $\{ X_w |_v, \; v \geq w\}$ forms an affine open cover of the Schubert variety $X_w$.
Let $T$ denote the maximal torus $T \cong (\C^x)^n$, and $(Gr_k \C^n)^T$ denote the fixed points under the $T$ action. Let $\lambda \in (Gr_k \C^n)^T \cong S_n/(S_k \times S_{n-k} \cong \binom{n}{k})$. Equivalently, $\lambda=(\lambda_1,\ldots, \lambda_k) \in \binom{[n]}{k}$, where we use the notation $[n] = \{1,\ldots,n\}$. Let $\C^{[i,\ldots,j]}$ be the subset of $\C^n$ where only the entries in places $i$ to $j$ (inclusive) are nonzero. The equations defining the Schubert variety in $Gr_k \C^n$, where $\lambda$ is considered as a bit string, are $$X_\lambda = \{V \;|\; \text{dim}(V \cap \C^{[1,\ldots,i]}) \geq \# 1 \text{'s in } \lambda \text{ in } [1,\ldots,i] \}$$ and those of the open Schubert cell are
$$X_\lambda^\circ = \{V \;|\; \text{dim}(V \cap \C^{[1,\ldots,i]}) = \# 1 \text{'s in } \lambda \text{ in } [1,\ldots,i] \}.$$
The equations defining the opposite Schubert variety in $Gr_k \C^n$ are $$X^\mu = \{V \;| \;\text{dim}(V \cap \C^{[i,\ldots,n]}) \geq \# 1 \text{'s in } \mu \text{ in } [i,\ldots, n] \}$$
and those of open opposite Schubert cell are
$$X^\mu_\circ = \{V \;| \;\text{dim}(V \cap \C^{[i,\ldots,n]}) = \# 1 \text{'s in } \mu \text{ in } [i,\ldots, n] \}.$$

In this thesis, we will look at varieties from the matrix perspective, where the $V$'s are represented by by $n \times n$ matrices of rank at most $k$, whose row span yields the subspace $V$. A Schubert patch corresponds to setting a $k$-subset of the columns equal to the columns of the identity matrix $I_k$.

In particular, the conditions imposed on a Schubert variety correspond to rank conditions on terminal intervals of columns, and on an opposite Schubert variety to rank conditions on initial intervals of columns. This implies that Richardson varieties are defined by the intersection of terminal and initial rank conditions, but sometimes this together with the condition that $dim(V)=k$ gives interval conditions that are neither initial nor terminal.

\begin{Exa} In $M_{2 \times 3}$, the Schubert variety $X_{3124}$ gives the rank condition $rank[1,2] \leq 1$ and the opposite Schubert variety $X^{1423}$ gives the condition $rank[2,3] \leq 1$. Then the ideal generated by these two conditions is reducible to the conditions $rank[1,2,3] \leq 1$ or $rank[2] = 0$. However, the Richardson variety $X_{1423}^{3124}$ requires only the latter.
\end{Exa}

As in \cite{Fulton}, a {\bf matrix Schubert variety} is defined for $\pi \in S_n$ by $$\overline{X_\pi} = \overline{B_{-}\underline{\pi} B_{+}} \subseteq M_n \C$$  The Bruhat order on $S_n$ corresponds to reverse containment on matrix Schubert varieties:  $\pi \leq \rho$ if and only if $\overline{X_\pi} \supseteq \overline{X_\rho}$.

\subsection{Stratifications Constructed From Hypersurfaces}\label{sec:strats}

 Given a polynomial $f$, we start with the hypersurface $f=0$ and construct varieties by taking components, intersecting them, taking unions and repeating, as in \cite{Kfs}.

\begin{Thm}\cite[Theorem 4]{Kfs} \label{thm:initprod}
  Let $f \in \Z[x_1,\ldots,x_n]$ be a degree $k$ polynomial
  whose lexicographically first term is (a $\Z$-multiple of)
  a product of $k$ distinct variables.

  Let $Y$ be one of the schemes constructed from the hypersurface
  $f=0$ by taking components, intersecting, taking unions, and repeating.
  (Or more generally, let $Y$ be compatibly split with respect to the
  splitting $\Tr(f^{p-1}\bullet)$.)
  Then $Y$ is reduced over all but finitely many $p$, and over $\Q$.

  Let $\init Y$ be the lex-initial scheme of $Y$. Then (away from those $p$)
  $\init Y$ is a Stanley-Reisner scheme.
  \end{Thm}

\begin{Exa} Using this method, we can construct a poset of matrix Schubert varieties $P=\{\overline{X_\pi}, \pi \in S_n\}$ as follows. Let $$f=\prod_{i=1}^{n-1} \text{det} M_{[i \times i]} \in \Z[x_{11},\ldots, x_{nn}]$$ where $M_{[i \times i]}$ denotes the $i \times i$ northwest submatrix. Start with $\{f=0\} \subseteq M_n$. Then decompose this subscheme, intersect the pieces, take unions, and repeat. This process produces all and only matrix Schubert varieties by \cite[\S 8.2]{Kfs}. The top element is the whole space and covering relations are given by containment.  See Figure \ref{fig:matrixSch}. The basic elements are indexed by those $\pi$ that are bi-Grassmannian. That is, their diagrams have only one rectangle, giving just one essential box and thus one rank condition.

\begin{figure}[htbp]
\begin{center}
\scalebox{0.7}{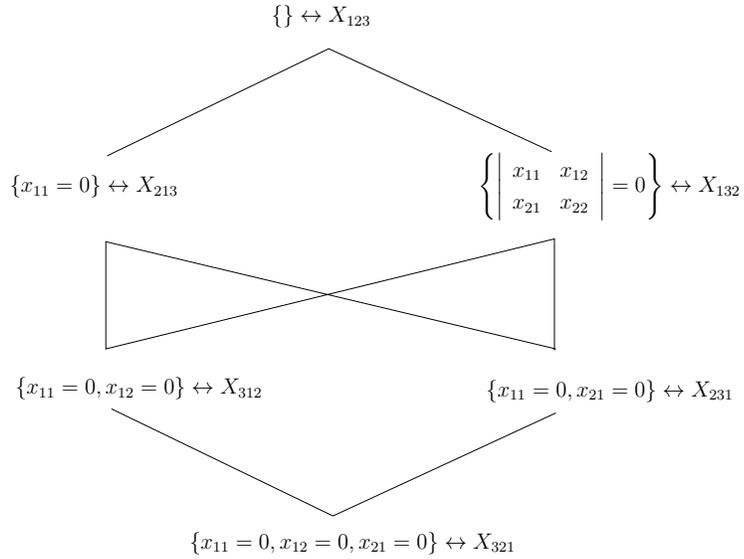}
\caption{Example of matrix Schubert construction for $n=2$.}\label{fig:matrixSch}
\end{center}
\end{figure}
\end{Exa}

In the case of matrix Schubert varieties, the intersections are always reduced. Note that this is not true for all $f$: for example, if $f=y(y-x^2)$, the intersection of the components $y=0$ and $y=x^2$ is not reduced due to the double point at $y=0$.

\begin{Thm}\cite{Sturm} Let $\pi$ be biGrassmannian. The determinants defining $\overline{X_{\pi}}$ are a Gr{\"o}bner basis for any antidiagonal term order. \end{Thm}

Fulton proved in \cite{Fulton} that concatenating the ideals of the set of basic elements defined by particular $\pi$ give the ideal for $\pi$. The following theorem gives the same idea for Gr{\"o}bner bases. We will apply it to the case of $G=GL_n$.

\begin{Thm} \cite[Theorem 7]{Kfs} \label{thm:kl}
  Fix $v\in W$, and a reduced word $Q$ for $v$.
  Then the function $f$ on $\mathbb{A}^{\ell(v)}$ defined by
  $$ f(c_1,\ldots,c_{\ell(v)}) :=
  \prod_\omega m_\omega(\tilde\beta_Q(c_1,\ldots,c_{\ell(v)}))\qquad$$
 where $\omega$ is ranging over $G$'s fundamental weights,
  is of degree $\ell(v)$, and its lex-initial term is
  $\prod_i c_i$.

  Under the identification of $\mathbb{A}^{\ell(v)}$ with $X^v_\circ$,
  the divisor $f=0$ is the preimage of $\bigcup_\alpha X_{r_\alpha}$.
  By decomposing and intersecting repeatedly, we can produce
  all the other $X_{w\circ}^v$ from this divisor.
  If $I_w^Q$ is the ideal in $\Q[c_1,\ldots,c_{\ell(v)}]$
  corresponding to $X_{w\circ}^v$, then $\init I_w^Q$ is Stanley-Reisner.

\pagebreak

  We can produce a Gr\"obner basis for $I_w^Q$ by concatenating
  Gr\"obner bases for $I_{w'}^Q$, with $w' \leq w$ in Bruhat order,
  and $w'$ basic in {\em opposite} Bruhat order on $W$.
  (The basic elements of opposite Bruhat orders were computed
  in \cite{LS,GeckKim}.)
\end{Thm}

In particular, $\init I_w^Q$ is the Stanley-Reisner ideal of a particular simplicial complex. In this thesis, we will consider a richer situation, where we start with the union over Schubert divisors, $D=\bigcup_{i=1}^n X_i$, inside $Gr_k \C^n$. This intersect/decompose/repeat process will then yield positroid varieties $Y$, as well as parallel results about Gr{\"o}bner bases and Stanley-Reisner ideals.

\section{Bott-Samelson Conditions in Opposite Bruhat Cells}\label{sec:oldtermorder}
Take the case $G=GL_n$ and let $Q$ be a reduced word in $S_n$ and $\Pi Q = \pi$. As in \cite[\S 3]{Kfs}, associated to $Q$ is a {\bf Bott-Samelson manifold} $$BS^Q := P_{\alpha_1} \times^B \cdots \times^B P_{\alpha_{\ell(\pi)}}/B$$
and birational map $\beta_Q: BS^Q \twoheadrightarrow X^\pi$, taking
$[p_1,\ldots,p_{\ell(\pi)}] \mapsto \left( \prod_{i=1}^{\ell(\pi)} p_i\right) B/B$.
In particular, we can use $\beta_Q$ to define an isomorphism from affine space to the opposite Bruhat cell \begin{eqnarray*}
\mathbb{A}^{|Q|} & \xrightarrow{~} &X_\circ^\pi \\
(c_1,\ldots,c_{|Q|}) &\mapsto& \left( \prod_{i=1}^{\ell(\pi)} (e_{\alpha_i}(c_i) \tilde r_{\alpha_i}) \right) B/B
\end{eqnarray*}
where the matrix $e_{\alpha_i}(c_i) \tilde r_{\alpha_i}$ represents the identity matrix modified so that the $2 \times 2$ block starting at $(i,i)$ has been replaced with $\left(
                                                                                                                          \begin{array}{cc}
                                                                                                                            c_i & -1 \\
                                                                                                                            1 & 0 \\
                                                                                                                          \end{array}
                                                                                                                        \right)$.

To calculate the matrix entries of $\beta_Q(c_1 \ldots c_{|\alpha|})$, we draw the wiring diagram for $Q$, labeling the cross that executes $s_i$ with the variable $c_i$. Then, we read paths left to right, where at each cross we can choose whether to follow the path southwest to northeast or northwest to southeast through the cross, or from northwest we can ignore the cross and go northeast. We assign the following weights: southwest to northeast has weight 1, northwest to southeast has weight -1, and northwest to northeast has weight $c_i$. We make a matrix where the $(i,j)$ entry is the sum over the weights of the paths from $i$ to $j$. Note that the individual weights are only homogeneous if $Q$ is 321-avoiding.

\begin{Exa} Let $Q= 12312$. Figure \ref{fig:akq} shows the calculation of the matrix $e_{\alpha_i}(c_i) \tilde r_{\alpha_i}$. One can check that this indeed matches the product of the individual matrices associated to each transposition. Then we get the following determinants:
\begin{itemize}
\item[$i=1$:] $c_1 c_3 - c_2 \ra \init = c_1 c_3$
\item[$i=2$:] $(c_1 c_3 - c_2)(-c_5) - c-3 (c_4 - c_1 c_5) \ra \init = c_2 c_5$
\item[$i=3$:] $c_4 \ra \init = c_4$
\end{itemize}

\begin{figure}[htbp]
\begin{center}
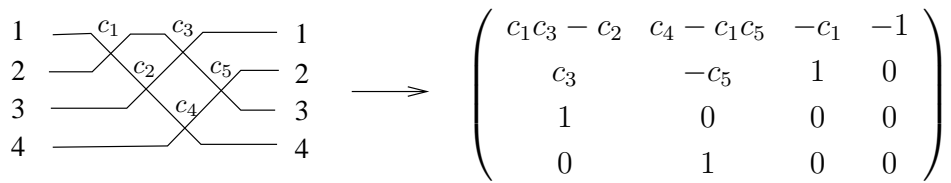
\caption{For $Q=12312$, $v=4312$. The product of the initial terms of the northwest determinants is $\prod_i c_i$. }
\label{fig:akq}
\end{center}
\end{figure}

Then, the product of the initial terms is $c_1 c_2 c_3 c_4 c_5$.

\end{Exa}

\begin{Thm}\cite[Theorem 7]{Kfs} For the matrix given by $\beta_Q(c_1 \ldots c_{|\alpha|})$, $$\init \left(\prod_{i} i \times i \text{ determinants} \right)= \prod_i c_i.$$ \end{Thm}

\subsection{Kazhdan-Lusztig Varieties}
We define a {\bf Kazhdan-Lusztig variety} as $X_{w^\circ}^v := X_w \cap X_\circ^v$. We can obtain Kazhdan-Lusztig varieties from Schubert patches by factoring out a vector space, and results about Schubert patches are often simpler when considered on Kazhdan-Lusztig varieties.
\begin{Lem}\cite[Lemma A.4]{KL} For $X_\pi \subseteq G/B$, $$X_\pi |_\lambda \cong (X_{\pi^\circ}^\lambda)\times (X_\lambda^\circ) \cong (X_\nu^\circ) \times (X_\lambda \cap X^\nu_\circ)$$ \end{Lem}
\noindent where, in the last term, $X_\nu^\circ$ tells us about the terms above $\lambda$ in the poset, and $X_\lambda \cap X^\nu_\circ$ tells about the terms below.
The stratification $X_{w,\circ}^v = \amalg_{x \geq w} X_{w,\circ}^{v,\circ} $ of these varieties is ``generated" by $X_\circ^v \cap \{X_{s_\alpha}\}$, for $s_\alpha$ a simple reflection. Let $f$ be the product of determinants of $k \times k$ submatrices in the upper left and lower right corners of the matrix. Then this stratification can be calculated starting with the hypersurface $f=0$, then intersecting, decomposing, and repeating. One reason that this stratification is of interest is that the closed strata are the compatibly Frobenius-split subvarieties, a fact we neither prove nor use.

\subsection{Positroid Varieties}\label{sec:posvars}

The permutation matrix for a bounded juggling pattern $f$ is a $\Z \times \Z$ matrix with a $1$ in row $i$, column $i+f(i)$. Define a {\bf diagram} crossing out all boxes {\it strictly} south or west of each 1. Note that all 1's are in a strip between the diagonals $j=i$ and $j=i+n$, with period $n$.
Then the corresponding {\bf positroid variety} is defined as
\begin{displaymath}\Pi_f=GL_k \setminus \bigg\{  M  \subseteq M_{k,n} \mid \rank([i,j]) \leq | [i,j]| - \# \textrm{ 1's southwest of } (i,j), i\leq j \leq i+n \bigg\}\end{displaymath}
where $\rank([i,j])$ denotes the rank of the submatrix defined by columns $i$ to $j$, cyclically.
Then the corresponding {\bf open positroid variety} is defined as
\begin{displaymath}\Pi_f^\circ=GL_k \setminus \bigg\{  M  \subseteq M_{k,n} \mid \rank([i,j]) = | [i,j]| - \# \textrm{ 1's southwest of } (i,j), i\leq j \leq i+n \bigg\}\end{displaymath}
 We get a condition for each point $(i,j)$, $i \leq j$, but some of these conditions imply the others.  It is sufficient to just consider the  {\bf essential set} of this diagram, where a condition on an interval $[i,j]$ being essential means that if you shrink the interval, the rank condition stays the same; when you enlarge the interval, the rank goes up. Graphically, the essential set is the northeast corners of the bounded regions in the diagram. This construction gives {\it cyclic} rank conditions, and we need only consider one period because the diagram repeats every $n$.

\begin{figure}[htbp]
\begin{center}
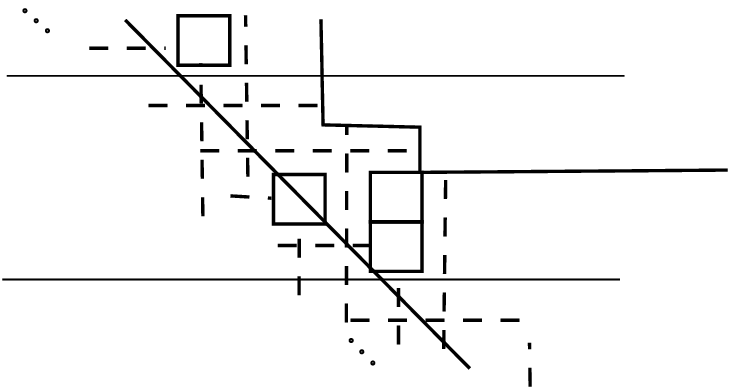
\caption{Diagram for the juggling pattern 3401.}
\label{fig:3401}
\end{center}
\end{figure}

\begin{Exa}\label{ex:positroid} Let $f=3401$. In order to get the associated positroid variety $\Pi_f$, we draw the diagram as in Figure \ref{fig:3401}. Then there are two essential boxes, $(3,3)$ and $(3,5)$, which give the rank conditions $\rank[3,3] \leq 0$ and $\rank[3,1] \leq 1$. Note that the condition from the non-essential box $(4,5)$, $\rank[4,1] \leq 1$, is implied by $\rank[3,1] \leq 1$.
\end{Exa}

\begin{Lem} Rank conditions on the column interval $[i,j]$ correspond to the number of arcs entering the range $[i,j]$, or equivalently to the number of arcs leaving the range $[i,j]$. \end{Lem}
\begin{proof} This follows from the fact that $\# \{ \textrm{ 1's southwest of } (i,j), i\leq j \leq i+n\}$ is equal to the number of throws starting weakly after $i$ and ending weakly before $j$. Then,

\begin{eqnarray*} \rank[i,j] &\leq &|[i,j]| - \#\{\text{throws starting after } i  \text{ and ending before } j \} \\
&=& \#\{\text{throws starting in } [i,j] \text{ and ending strictly after } j \}\\
&=& \#\{\text{throws starting strictly before } i \text{ and ending in } [i,j] \}
\end{eqnarray*}
\end{proof}

\begin{Thm}\cite{KLS} Patterns of totally non-negative matrices correspond to bounded juggling patterns. Furthermore, every juggling pattern arises this way. \end{Thm}

In Figure \ref{fig:gr24plus} we show the decomposition of $Gr_k \C^n$ into open positroid varieties for $n=4$ and $k=2$, indexed by juggling patterns. We will call a juggling pattern {\bf basic} if its corresponding positroid variety has only a single rank condition. This occurs when the diagram for $f$ has a single essential box. For a pattern of length $n$ with $k$ balls, and single rank condition $\rank [i,j] \leq r$, this corresponds to a juggling pattern of the form
$$ f = k^{(i-1)-(k-r)} (|[i,j]| - r+k)^{k-r} r^{ |[i,j]| - r} k^{n-j+r}$$
That is, a series of $k$-throws, followed by a set of high throws, then a set of low throws, then back to $k$'s.

\begin{figure}[htbp]
\begin{center}%
\scalebox{0.25}{\rotatebox{90}{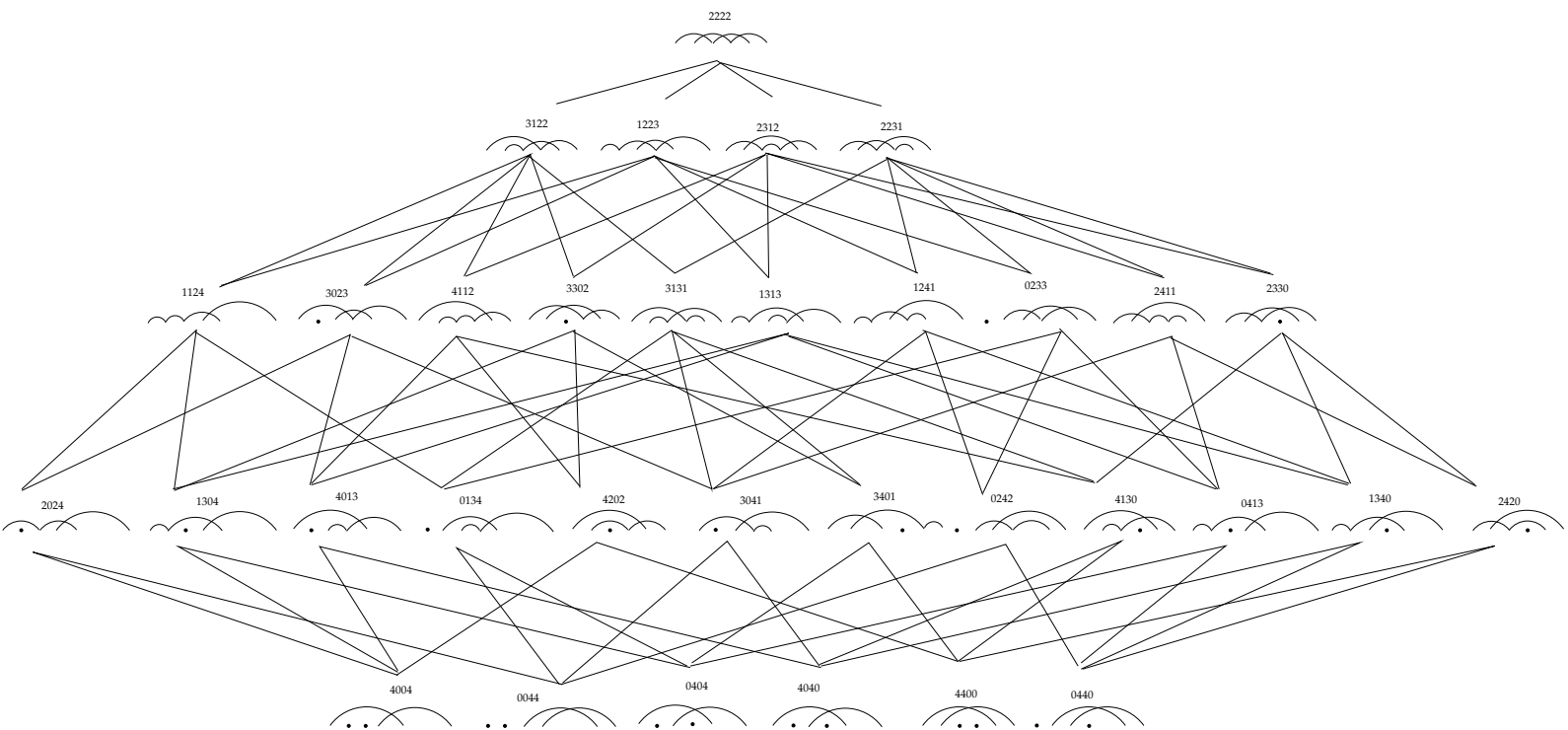}}
\caption{The poset of cells of the totally nonnegative part of $Gr_2 \C^4 (\R)$, with cells indexed by juggling patterns.}\label{fig:gr24plus}
\end{center}
\end{figure}

We study open patches on positroid varieties, indexed by $\lambda \in \binom{[n]}{k}$. As in the case of Schubert varieties, we denote the patch on $\Pi_f$ centered at $\lambda$ by $\Pi_f |_\lambda$. (There is no analogue of the lemma of \cite{KL}.) Recall that in the matrix description, $\lambda$ is a $k$-subset of the columns and we set column $\lambda_i$ to the $i^{th}$ column of the identity matrix $I_k$, leaving the rest of the entries free. Note that if we pick $\lambda$ such that setting columns $\lambda_i$ to the identity gives a matrix that violates the rank conditions from $f$, then $\lambda \notin \Pi_f$ and $\Pi_f|_\lambda$ is empty. Non-empty patches centered at $\lambda$ correspond to the $T$-fixed points on the Grassmannian that lie inside a particular positroid variety $\Pi_f$.
\begin{cont_ex}{\bf\ref{ex:positroid}. }
For $\lambda=(1,2)$, we get rank conditions $rank[3,3] \leq 0$ and $\rank[3,1] \leq 1$ so the variety is composed of matrices of the form $$\left(
                                  \begin{array}{cccc}
                                    1 & 0 & 0 & \star \\
                                    0 & 1 & 0 & 0 \\
                                  \end{array}
                                \right)$$
where the entry $\star$ is free.
\end{cont_ex}

Given a matrix $M \in M_{k \times n}$ of rank $k$, we can determine which open positroid variety its row span lives in by the following method. For column $i$, look for the first column $j$ cyclically after $i$ such that column $i$ is dependent on columns $i+1, \ldots,j \mod n$. Then to this we associate a throw from $i$ to $j$, to construct a list of throws, $f$, indexing the positroid.
\begin{Lem} As constructed, this $f$ is a juggling pattern. \end{Lem}
\begin{proof}
$f$ satisfies $f(i) \in [i,i+n]$ and $f(i+n) = f(i)+n$. Suppose $v_{i_1}$ and $v_{i_2}$ both have $v_j$ as the first dependent vector for $i_1 < i_2 < j$, then $0 = \sum_{k=i_1}^{j} c_k v_k$ where $c_{i_1} \neq 0$ and $c_j \neq 0$ and  $0 = \sum_{k=i_2}^{j} c'_k v_k$ where $c'_{i_2} \neq 0$ and $c'_j \neq 0$. Then we can combine these two sums into $0 = \sum_{k=i_1}^{j} (c_k+c'_k) v_k$ where $(c_{i_1}+c'_{i_1}) \neq 0$ and  $(c_{i_2}+c'_{i_2} )\neq 0$ ($c_j$ may equal 0). Then $v_{i_2}$ is dependent on $v_{i_1}$, contradicting the first dependent vector choice. Then $f$ is one-to-one. $f$ is onto by periodicity.
\end{proof}

All matrices in a given positroid variety have the same cyclic rank structure, so we choose to index the variety by the associated juggling pattern.

\section{Affine Flags}\label{sec:affineflags}

Our reference for this section is \cite{PS}.

A {\bf lattice} L is a linear subspace of $\C[[t^{-1}]][t]$, where the codimension of $(L \cap \C[[t^{-1}]][t])$ in $L$ is finite, representing the number of balls in the air, and the codimension of $L \cap \C[[t^{-1}]][t]$ in $\C[[t^{-1}]][t]$ is also finite, representing the balls in the air traveling backwards in time (antiballs). We define the difference in these codimensions to be the {\bf index}, and it equals the net number of balls in the pattern. See Figure \ref{fig:index}. We call $t^c \in \Z^n$ {\bf translation elements}.

\begin{figure}[htbp]
\begin{center}
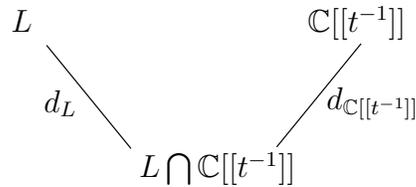
\caption{The index of $L$ is $d_L-d_{\C[[t^{-1}]]}$.}
\label{fig:index}
\end{center}
\end{figure}

We define the $k^{th}$ component of the {\bf affine Grassmannian} as $$\ag{k,n} = \{L \mid L \textrm{ lattice, } t^{-n}L \subseteq L, \textrm{ and } index(L)=k \}$$ and the { \bf affine flag manifold} as $$\af{k,n}=\{ (\ldots, L_0, L_1, \ldots, L_n, L_{n+1}, \ldots) \mid L_i \in \ag{k,n}, L_i=L_{i+n}, t^{-1} L_{i} \subseteq L_{i+1} \}$$ Note that the last condition implies that $t^{-n} L_i \subset L_{i+n}$. Our notation differs from the usual definition of a lattice, $L \subset \C[[z^{-1}]][z] \otimes V^{(n)}$ with the condition that $z^{-1} L \subseteq L$. We can correspond these definitions by $t^{ni+j}\longleftrightarrow z^i \otimes \overrightarrow{v_j} $. One benefit of the usual definition is to see that the affine flag is just $G/B$, where the relevant $G$ is $GL_n( \C[[z^-1]][z] )$, acting on $\C[[z^-1]][z] \otimes \C^n$ and we restrict $B$ and $P$ to matrices with only negative powers of $z$. Then we preserve the lattice condition that $z^{-1} L \subseteq L$. This lets us then use the Bruhat and opposite Bruhat decompositions from $G/B$.

Given $\lambda \in \binom{[n]}{k}$, we can associate to it a list of states $\Lambda= (\Lambda_1,\ldots,\Lambda_n)$, where $\Lambda_1$ has $\times$'s in the entries of $\lambda$ and $-$'s elsewhere, and $\Lambda_i$ is the $i^{th}$ rotation of $\lambda$.  We let {\bf $t_\lambda$} be the corresponding flag, where we construct the lattice as follows: $\times$ in the $j^{th}$ entry of $\lambda_i$ goes to the term $\C[[t^{-1}]] \oplus \C \cdot t^j$.
\begin{Exa} Let $\lambda = (1,3)$, then we also write it as a state $\Lambda_1 = \times - \times - $ and as a lattice, $L(\Lambda_1) = \C[[t^{-1}]] \oplus \C \cdot t \oplus \C \cdot t^3 $. Then, $\Lambda = (\times - \times -,  - \times - \times,\times - \times -,  - \times - \times )$. As a flag, $t_\lambda = (\C[[t^{-1}]] \oplus \C \cdot t \oplus \C \cdot t^3,  \C[[t^{-1}]] \oplus \C \cdot t^2 \oplus \C \cdot t^4, \ldots)$. \end{Exa}
We define the following subset for $\mu$ as a list of states:
\begin{displaymath} \ag{}^{\mu} =\{L \mid \textrm{dim} (L_i/(L_i \cap t^m \C[[t^{-1}]])) \geq \textrm{dim} (\mu_i/(\mu_i \cap t^m \C[[t^{-1}]])) \}.\end{displaymath}
In juggling terms, the right hand side of the inequality is equal to the number of balls landing weakly after $m$ in $\mu_i$. These conditions define a finite-dimensional closed subset. We also will use the open version
\begin{displaymath} \ag{\circ}^{\mu} =\{L \mid \textrm{dim} (L_i/(L_i \cap t^m \C[[t^{-1}]])) = \textrm{dim} (\mu_i/(\mu_i \cap t^m \C[[t^{-1}]])) \}\end{displaymath}

 We also have the subset
\begin{displaymath} \ag{\lambda} = \{L \mid \textrm{dim} (L_i \cap t^m \C[t]) \geq \textrm{dim} (\lambda_i \cap t^m \C[t]) \}\end{displaymath}
 where the right hand side is equal to the number of balls landing weakly after $m$ in $\lambda_i$. These conditions define a finite-{\em codimensional} closed subset. These spaces are opposite Schubert and  Schubert by $\S 8.4$ in \cite{PS}. We also have the open
 \begin{displaymath} \ag{\lambda}^{\circ} = \{L \mid \textrm{dim} (L_i \cap t^m \C[t]) = \textrm{dim} (\lambda_i \cap t^m \C[t]) \}\end{displaymath}
  We let $\ag{\lambda}^\mu = \ag{\lambda} \cap \ag{}^\mu$. We can extend this to affine flags via the embedding of $\af{k,n}$ inside the product of copies of the $\ag{k,n}$.

\chapter{Main Results: Corresponding the Big Cells to Affine Flags}\label{ch:new}

\section{New Term Order and Affine Permutation Construction}\label{sec:newtermorder}
We are interested in Schubert patches on positroid varieties, considered as sets of matrices. For a Schubert patch on the Grassmannian, we consider $k \times n$ matrices with an associated $\lambda \in \binom{[n]}{k}$. We will construct corresponding $k \times (n-k)$ matrices with a record of $\lambda$, and put a term order on the variables. Define the {\bf distinguished path} to represent $\lambda$ as follows. Let $M$ be the $k \times n$ matrix where we set the $\lambda_i^{th}$ column equal to the $i^{th}$ column of the identity matrix $I_k$, and leave all other entries free. Draw a line starting from the northwest corner along the grid lines, where the line moves east until it passes above a 1, at which point it moves one unit south.

\begin{figure}[htbp]
\begin{center}
\scalebox{0.7}{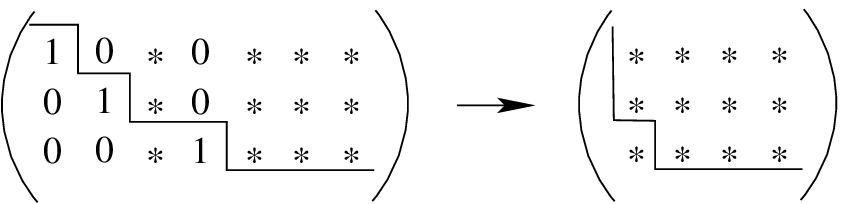}
\caption{Collapsing the matrix with the distinguished path.}
\label{fig:collapse}
\end{center}
\end{figure}

As in Figure \ref{fig:collapse}, we collapse the $k \times n$ matrix to a $k \times (n-k)$ matrix by removing the identity columns, so that the distinguished path records where those columns were. Starting with the entry in the collapsed matrix corresponding to the $(k,1)$ entry in the original matrix, we label {\bf split antidiagonals} in the order $k,1,2,\ldots,k-1$ such that they \begin{enumerate}
\item skip over the columns of $I_k$,
\item do not cross the distinguished path, and
\item start above the distinguished path, then cycle around to the terms below.
\end{enumerate}

\begin{figure}[htbp]
\begin{center}
\scalebox{0.8}{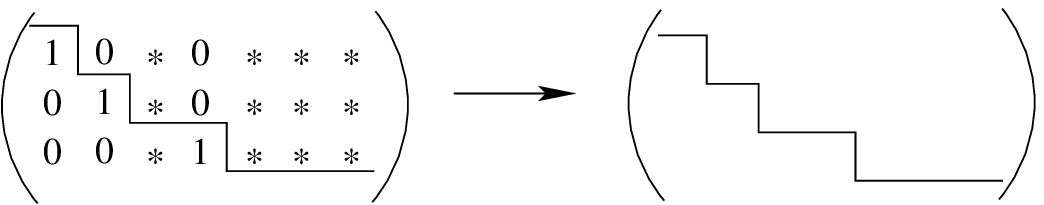}
\caption{New term order on the original $k\times n$ matrix for $\lambda = (1,2,4)$ for $k=3$ and $n=7$.}
\label{fig:termorder}
\end{center}
\end{figure}

In Figure \ref{fig:termorder}, we underline the antidiagonal labeling so as not to confuse the 1's in the identity columns with the first antidiagonal. We choose the reverse lexicographic term order on these matrix entries. Note that even though our reverse lexicographic order is only a partial order, the following lemma gives us the fact that our reverse lexicographic first terms will be monomials.

 \begin{Lem} With this new order, $$\init \left(\prod_{i=1}^n \; \det M_{[i,i+k-1]}\right) = \prod_{i,j} \; a_{ij}$$ where $[M]_{j,k}$ denotes the submatrix of $M$ given by columns $j$ to $k$, cyclically. \end{Lem}

\begin{proof} Applying the reverse lexicographic order to a given square submatrix is equivalent to crossing out the lowest numbered boxes, then the second lowest, and so on, until only one weight remains. The determinant of columns $[i,i+k-1]$ under this order picks out the antidiagonal elements labeled with $(i-1) \mod n$. Then each entry appears in only one determinant, and so the first monomial in the product under this order will equal $\prod_{i,j} \; a_{ij}$.
\end{proof}

 Draw copies of the collapsed matrix in a diagonal, to create a northwest-southeast {\bf strip} between the paths.
 We want to associate to this strip an affine permutation, $\pi_\lambda$. We label each unit step on our line with an integer,
$\ldots,-2, -1,0,1,2,\ldots$, one for each unit step east and south. Using the labeling for our term order, associate entry $k$ with the transposition $s_k$. We will refer to a {\bf block} in the strip as a unit that contains each original box just once.

\begin{figure}[htbp]
\begin{center}
\scalebox{0.8}{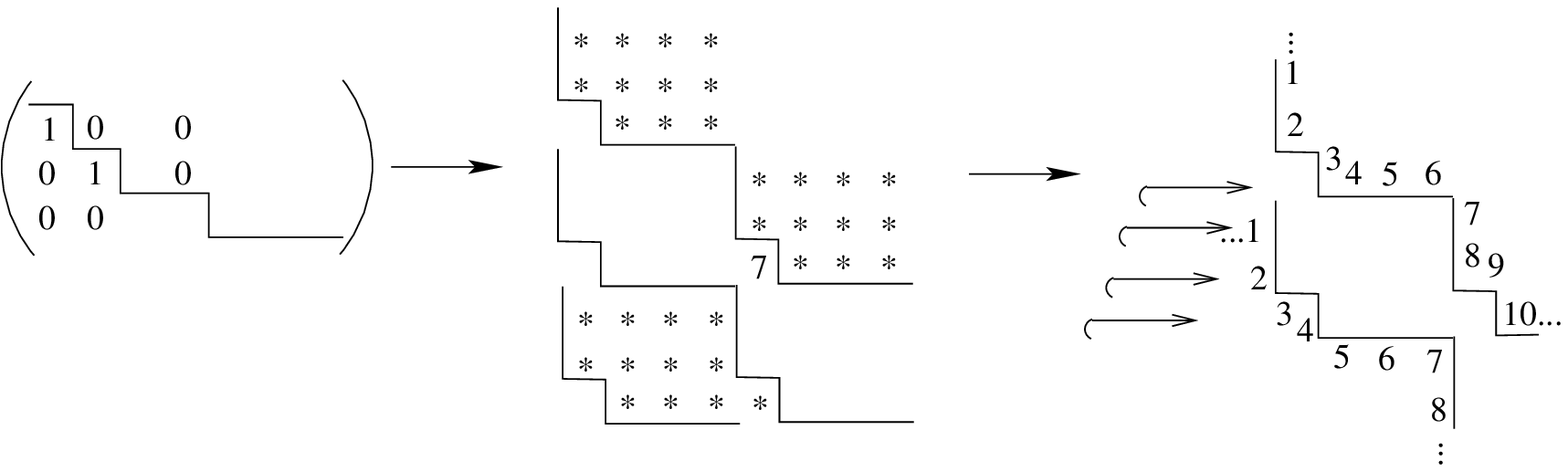}
\caption{Creating the strip for $\lambda = (1,2,4)$ for $k=3$ and $n=7$, and $Q_\lambda=456\;2345\;1234\;7$.}
\label{fig:strip}
\end{center}
\end{figure}

Then we define the word $Q_\lambda$ by reading west to east, south to north, inside single a block. See Figure \ref{fig:strip}. The permutation $\pi_\lambda$ is determined by $Q_\lambda$ operating on the affine permutation $[k+1,k+2,\ldots, k+n]$. That is, $\pi_\lambda =Q_\lambda [k+1,k+2,\ldots, k+n]$.

\begin{Lem} The permutation $\pi_\lambda$ resulting from this construction is 321-avoiding.\label{lem:321}\end{Lem}
\begin{proof} Any two occurrences of $s_i$ must be on the same antidiagonal, thus occur in a block of the form $\left(
           \begin{array}{cc}
             s_{i-1} & s_i \\
              s_i &  s_{i+1} \\
           \end{array}
         \right)$
\noindent yielding a substring in $Q_\lambda$ of the form $s_i,  s_{i+1},\ldots,s_{i-1},s_i$. Since any pair of adjacent transpositions is not commutative, it is not possible to get a substring of the form $s_i s_{i \pm 1} s_i$. The statement follows from \S \ref{sec:perms}.
\end{proof}

\begin{Prop} The permutation $\pi_\lambda$ corresponds to the juggling pattern \begin{displaymath}
f(i) = i+\left\{ \begin{array}{ll}
n & \textrm{if } i \in \lambda,\\
0 & \textrm{otherwise.}\\
\end{array} \right.
\end{displaymath}
\end{Prop}

\begin{proof}
Lemma \ref{lem:321} implies that all the reduced words corresponding to $\lambda$ are fully commutative, thus we can make a heap by simply rotating a single block of the strip by $45^\circ$ counterclockwise. We create the associated wiring diagram by replacing each transposition with a cross. We cyclically shift the wiring diagram to get chains of crosses from $(\lambda_i+(n-k-1),\ldots,\lambda_i)$ for each $\lambda_i \in \lambda$, reading top to bottom in the wiring diagram. Then elements of the form $\lambda_i+k$ are moved $n-k$ spots to the right, while the rest of the elements are moved $k$ spots to the left. Then
\begin{displaymath}
\pi_\lambda(i) = \left\{ \begin{array}{ll}
i+n & \textrm{if } i \in \lambda,\\
i & \textrm{otherwise.}\\
\end{array} \right.
\end{displaymath}
This corresponds to the desired juggling pattern, with $n$'s in the entries of $\lambda$ and 0's elsewhere. \end{proof}

\begin{Exa} Let $\lambda=(1,2,4)$. Then Figure \ref{fig:lambdajp} shows the construction of a heap and wiring diagram from the collapsed matrix.
\begin{figure}
\begin{center}
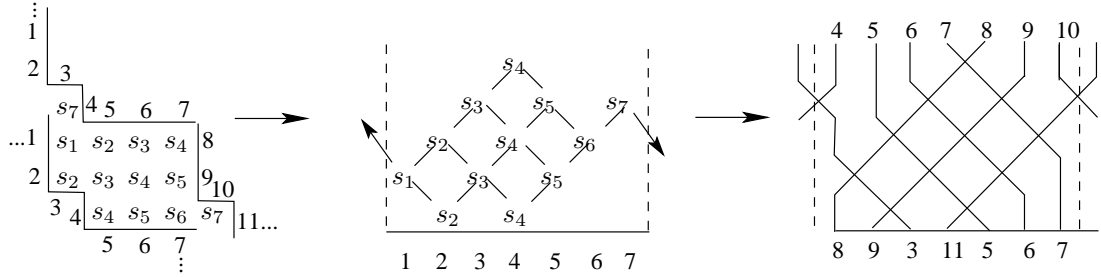
\caption{Corresponding $Q_{(1,2,4)}$ to the juggling pattern $(7,7,0,7,0,0,0)$.}
\label{fig:lambdajp}
\end{center}
\end{figure}
\end{Exa}

\section{Isomorphism of an Affine Kazhdan-Lusztig Variety and a Patch on a Positroid Variety}
We return to the big cell in $Gr_k \C^n$ centered at $\lambda$. This has a stratification given by the intersections with positroid varieties $\Pi_f$, and a term order on its (polynomial) coordinate ring defined in \S \ref{sec:newtermorder}. The affine opposite Schubert cell $\af{\circ}^{\pi_\lambda}$ is stratified by its intersections with $(\af{k,n})_{f}$, with $T$-fixed points $f$ given by virtual juggling patterns of length $n$ with $k$ balls.

\begin{Thm}[Main Theorem]\label{thm:mainthm}
Let $f$ be a bounded juggling pattern of length $n$ with $k$ balls, and let $\lambda \in \binom{[n]}{k}$. Then, $$\overline{\Pi_f}|_\lambda \cong (\af{k,n})_{f,\circ}^{t_\lambda}$$
Moreover, there is a correspondence between the stratifications $(Gr_k \C^n)|_\lambda$ and $ (\af{k,n})_{\circ}^{\pi_\lambda}$.
\end{Thm}

We prove this result in the next section, where it will suffice to show it holds for those $f$ defined by only one determinant condition (by \S \ref{sec:jp}). These are the elements of codimension 1 and generate the desired stratification, as all elements of higher codimension can be found by the intersect/decompose algorithm described in \S \ref{sec:strats}. We summarize our results in Figure \ref{fig:summary}, where the results on the right side can be found in \cite{Kfs} and the left and horizontal correspondences in this thesis.
\begin{figure}
\centerline{
\xymatrix{
\overline{\Pi_f}|_\lambda \ar@{^{(}->}[d] \ar[r]^{\cong} &\af{f,\circ}^{t_\lambda} \ar@{^{(}->}[d]\\
\overline{\Pi_{kk\dots k} }|_\lambda \ar[r] & \af{\circ}^{t_\lambda}\\
\mathbb{A}^{k(n-k)} \ar[r]^{=} \ar[u]^{\cong} & \mathbb{A}^{k(n-k)}\ar[u]^{\cong}_{\beta_{Q_\lambda}} \\
\text{my revlex order} \ar[r] & \text{lex order} \\
\prod_{i=1}^n \; \det M_{[i,i+k-1]} \ar[r] & p=p_{Q_\lambda}^{-1} \left(\bigcup_{i=1}^n X_{s_i}\right)
}
}
\caption{Summary of isomorphism proven in main theorem.}
\label{fig:summary}
\end{figure}

\subsection{Proof of Main Theorem, By Construction}

Given $\lambda =(\lambda_1,\ldots, \lambda_k) \in \binom{[n]}{k}$, we give an explicit construction of a family of affine flags. Recall we define $M$ as the $k \times n$ matrix with column $\lambda_i$ set equal to the $i^{th}$ column of the identity matrix, and the rest of the entries $a_{ij}$ free (for a total of $k \times (n-k)$ variables). We will associate the $i^{th}$ column of $M$ with $t^i$. We construct the $i^{th}$ lattice $L_i$ in the flag as follows:

\begin{enumerate}
 \item {\em Rotate:} Rotate the columns of $M$ $(i-1)$ times to the left, so that the previous column $i$ is now in the place of column 1.
\item {\em Clear out rows:} By construction, row $j$ has $a_{jl}=1$ for some $l$. Set all entries $a_{jm}=0$ for $m>l$. Call this matrix $M_i$.
\item {\em Create summands:} For each row $j$ in $M_i$, make polynomial $p_j(t)$ using the entries $a_{jl}$ as the coefficients for $t^j$.
\item {\em Construct $L_i$:} Direct sum the $p_j(t)$ components together with $\C[[t^{-1}]]$ to make the lattice $L_i$.
\end{enumerate}

\begin{Exa}\label{ex:mainthmex}
Let $\lambda = (1,2,4)$ and $n=7$.

Then, $M=\left(
                                              \begin{array}{ccccccc}
                                                1 & 0 & a_{13} & 0 & a_{15} & a_{16} & a_{17} \\
                                                0 & 1 & a_{23} & 0 & a_{25} & a_{26} & a_{27} \\
                                                0 & 0 & a_{33} & 1 & a_{35} & a_{36} & a_{37} \\
                                              \end{array}
                                            \right)$
Then,
$$M_1=\left(
                                              \begin{array}{ccccccc}
                                                1 & 0 & 0 & 0 & 0 & 0 & 0 \\
                                                0 & 1 & 0 & 0 & 0 & 0 & 0 \\
                                                0 & 0 & a_{3,3} & 1 & 0 & 0 & 0 \\
                                              \end{array}
                                            \right)$$
gives the polynomials $p_1(t)=t$, $p_2(t)=t^2$, and $p_3(t)=a_{33}t^3 + t^4$, and  $$L_1= \C[[t^{-1}]] \oplus \C \cdot (t) \oplus \C \cdot (t^2) \oplus \C \cdot (a_{33}t^3 + t^4)$$ Similarly,
                                            $M_2=\left(
                                         \begin{array}{ccccccc}
                                                0 & a_{13} & 0 & a_{15} & a_{16} & a_{17} & 1 \\
                                                1 & 0 & 0 & 0 & 0 & 0 & 0 \\
                                                0 & a_{33} & 1 & 0 & 0 & 0 & 0\\
                                              \end{array}
                                            \right)  
                                            \longleftrightarrow \\ \\
                                           L_2 = \C[[t^{-1}]] \oplus \C\cdot (t) \oplus \C \cdot(a_{33}t^2 + t^3 )\oplus \C \cdot(a_{13}t^2 + a_{15}^4 + a_{16}t^5+ a_{17}t^6 + t^7)
                                           $
                                             We also get $$L_3 = \C[[t^{-1}]] \oplus \C \cdot(a_{33}t + t^2 )\oplus \C \cdot(a_{13}t + a_{15}t^3 +\ldots+ t^6) \oplus \C \cdot(a_{33}t + \ldots+ t^7)$$
  $$ L_4 = \C[[t^{-1}]] \oplus \C \cdot(a_{33} + t)\oplus \C \cdot(a_{13} + a_{15}t^2 +\ldots+ t^5) \oplus \C \cdot(a_{33} + \ldots+ t^6) \oplus \ldots$$

and so on for $i=5$ to $7$.
\end{Exa}
 Given a juggling pattern $f$, we show the construction of rank conditions which must be satisfied by these matrices $M_i$. Recall that we need only consider juggling patterns with only one determinant condition, as explained after Theorem \ref{thm:mainthm}. We apply the inequalities from \S \ref{sec:affineflags}. Here, the only nontrivial requirement comes from the location of the last ball (at $k+1$): $$\textrm{dim}(L_i \cap t^{k+1} \C [t]) \geq 1.$$ This condition is equivalent to the statement that
$$\textrm{dim}(L_i/{\C[[t^{-1}]]} + t^{k+1}\C[t] / {t^n \C[t]}) \leq (\textrm{dim} (L_i) +\textrm{dim}(t^{k+1}\C[t]) -1)$$
which implies \begin{eqnarray*} n-1 &\geq& {\rank\left(
                                                                                                       \begin{array}{c|c}
                                                                                                         [M_i]_{1,k} & [M_i]_{k+1,n} \\
                                                                                                         \hline
                                                                                                         O & I_{n-k} \\
                                                                                                       \end{array}                                                                                                     \right)} \end{eqnarray*}

                                                                                                       or equivalently,
                                                                                                     $$k-1 \geq  \rank [M_i]_{1,k} \Rightarrow \textrm{det} [M_i]_{1,k} = 0.$$

\begin{cont_ex}{\bf\ref{ex:mainthmex}. }
Let $f=2333334$, corresponding to the single (affine flag) rank condition $\rank[1,2,3] \leq 2$. Given $\lambda = (1,2,4)$, we have
$$M_{[1,3]}=\left(
                                              \begin{array}{ccc}
                                                1 & 0 & a_{13}  \\
                                                0 & 1 & a_{23}  \\
                                                0 & 0 & a_{33}  \\
                                              \end{array}
                                            \right)$$
 On the positroid variety side, the rank condition from $f$ is $\rank M_{[1,3]} \leq 2$. Our affine flag construction gives us the condition $\rank ([M_1]_{[1,3]}) \leq 2$. Since the two matrices only differ by row reduction, we have that $\rank M_{[1,3]}=\rank ([M_1]_{[1,3]})$, so the two conditions are equivalent.
\end{cont_ex}

\begin{proof}[Proof of Theorem \ref{thm:mainthm}] We show that these lattices form a flag. Recall that
$$\af{k,n}=\{ (\ldots, L_0, L_1, \ldots, L_n, L_{n+1}, \ldots) \mid L_i \in \ag{k,n}, L_i=L_{i+n}, t^{-1} L_{i} \subseteq L_{i+1} \}$$

Each $L_i$ is a direct sum of the form $\C[[t^{-1}]] \oplus \C \cdot p_1(t) \oplus \ldots \C \cdot p_k(t)$, where $p_i$ is the polynomial associated with the $i^{th}$ row of $M_i$. This gives us that $d_{L_i} = k$, and $L_i \supseteq \C[[t^{-1}]]$ implies that $d_{\C[[t^{-1}]] } = 0$. Thus the index of $L_i$ is indeed $k$, so $L_i \in \ag{k,n}$. The condition $L_i=L_{i+n}$ is clearly satisfied by the cyclic structure. Now consider how $L_i$ is related to $L_{i+1}$ in our construction. For each component in the lattice $L_{i+1}$, compare rows in the matrices $M_i$ and $M_{i+1}$: the $j^{th}$ row of $M_{i+1}$ is the $j^{th}$ row of $M_i$ cyclically shifted to the left once, thus the entry in column 1 of $M_{i}$ are set to 0 in $M_{i+1}$, and the coefficient of $t^j$ in $L_{i+1}$ is equal to the coefficient $t^{j+1}$ in $L_i$. Thus, term by term, the condition $t^{-1} L_{i} \subseteq L_{i+1}$ is satisfied.

A basic juggling pattern $f$ corresponds to a condition of the form $\rank (M_{j,j+k-1}) \leq k-1$ for some $j \in [1,n]$. From our lattice conditions, we construct the conditions $\rank( [M_i]_{[j,j+k-1]}) \leq k-1$. The ranks of $M$ and $M_i$ are equal since $M_i$ can be obtained from $M$ by row reduction. Thus the conditions for the two sides of the isomorphism are equivalent.

\end{proof}

\section{Implications}

As in the $G=GL_n$ case, we construct combinatorial objects called affine pipe dreams, then show how they are related to subword complexes and antidiagonal complexes for $GL_n(\C[[t^{-1}]][t])$ Kazhdan-Lusztig varieties. They share many of the same properties as in the non-affine pipe dream case previously discussed.

\subsection{Affine Pipe Dreams}
 An {\bf affine pipe dream} is a diagram on the strip defined by $\lambda \in \binom{[n]}{k}$, as in \S \ref{sec:newtermorder}, filled with elbow $\textelbow$ and crossing $\textcross$ tiles. Note that we now allow crossing tiles in the entire diagram. As in the non-affine case, we say an affine pipe dream is {\bf reduced} if no two pipes cross twice. We will not consider non-reduced pipe dreams, and simply use {\it affine pipe dream} to mean {\it reduced affine pipe dream}. An affine pipe dream represents an affine permutation, where we follow the pipes from the bottom line, northeast to the top line. Note that if $\lambda = (1,\ldots, k)$ then we get a $k \times (n-k)$ rectangle, where the pipes are read from the west and south edges to the north and east edges, in that order. The chute and ladder moves described in \S \ref{sec:pipedreams} apply to affine pipe dreams as well.

\begin{Exa} For any $\lambda$, the $k$-ball cascade corresponds to the permutation $\pi(i)=i+k$ and the affine pipe dream with all elbow tiles. The siteswap with $n$'s in $\lambda_i$ spots and 0's elsewhere corresponds to the affine pipe dream with all crossing tiles. \end{Exa}

We construct an affine pipe dream for $\lambda$ and a juggling pattern $f$, by considering the associated cyclic rank conditions, as described in \S \ref{sec:posvars}. For each determinant, we pick out the antidiagonals as defined by our term order in \S \ref{sec:newtermorder}.

\begin{cont_ex}{\bf\ref{ex:mainthmex}. }\label{ex:apd} Let $\lambda=(1,2,4)$ and $\pi=5\;4\;6\;8\;7\;9\;10$, corresponding to $f= 4\;2\;3\;4\;2\;3\;3$. (The juggling pattern shifted back by $k=3$ is $f=1 \; -1 \; 0 \;1 \; -1 \; 0\; 0$.) The first rank condition $\rank[2,4] \leq 2$ gives $a_{33} = 0$. The second rank condition $\rank[5,7] \leq 2$ gives $\det M_{[5,6,7]} = 0$, and our term order picks out $a_{35}=0$, $a_{26}=0$, and $a_{17}=0$. Figure \ref{fig:124ex} shows the matrix form for the patch on the positroid variety. We represent the antidiagonal terms by \textcross tiles in the affine pipe dream. Figure \ref{fig:124pd} shows the corresponding affine pipe dreams, with elbows filled in. We show only one repetition of the block for clarity, but the reader should keep in mind that the affine pipe dream is actually an infinite strip.

\begin{figure}[htbp]
\begin{center}
\scalebox{0.7}{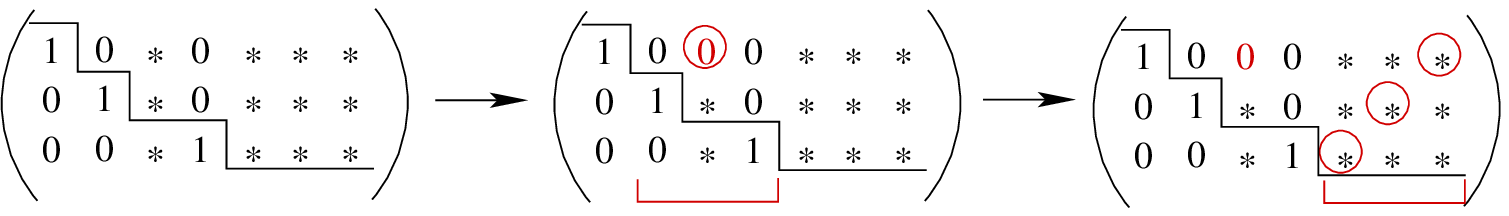}
\caption{Application of rank conditions for $f= 4\;2\;3\;4\;2\;3\;3$, with $\lambda=(1,2,4)$.}
\label{fig:124ex}
\end{center}
\end{figure}

\begin{figure}
\begin{center}
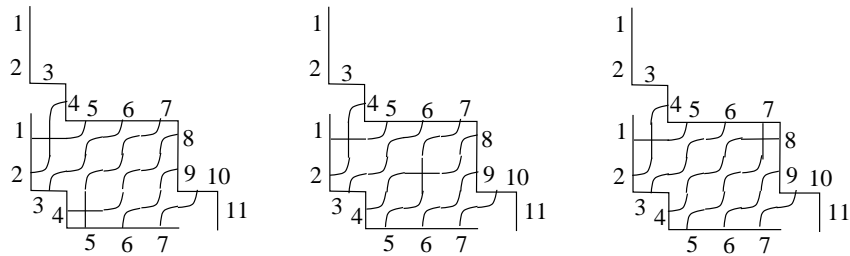
\caption{For $\lambda=(1,2,4)$, all the affine pipe dreams for the affine permutation $\pi=5\;4\;6\;8\;7\;9\;10$.}
\label{fig:124pd}
\end{center}
\end{figure}

\end{cont_ex}

\begin{Exa} For the same $\pi=5\;4\;6\;8\;7\;9\;10$, corresponding to $f= 4\;2\;3\;4\;2\;3\;3$, now let $\lambda=(4,6,7)$. The first rank condition $\rank[2,4] \leq 2$ gives $\det \left(\begin{array}{cc}
                                     a_{22} & a_{23} \\
                                     a_{32} & a_{33} \\
                                   \end{array}
                                 \right)
 = 0$, so we pick out antidiagonal terms $a_{32}=0$ and $a_{23}=0$. The second rank condition $\rank[5,7] \leq 2$ gives $a_{15}=0$. Thus we get two pipe dreams.
\end{Exa}

In Figure \ref{fig:twostrips}, we show two affine pipe dreams for the same permutation, but two different choices of $\lambda$. These show several repetitions within the infinite strip. For convenience, we will usually only draw one block, but the reader should keep in mind this is only a representative.

\begin{figure}[htbp]
\begin{center}
\scalebox{0.6}{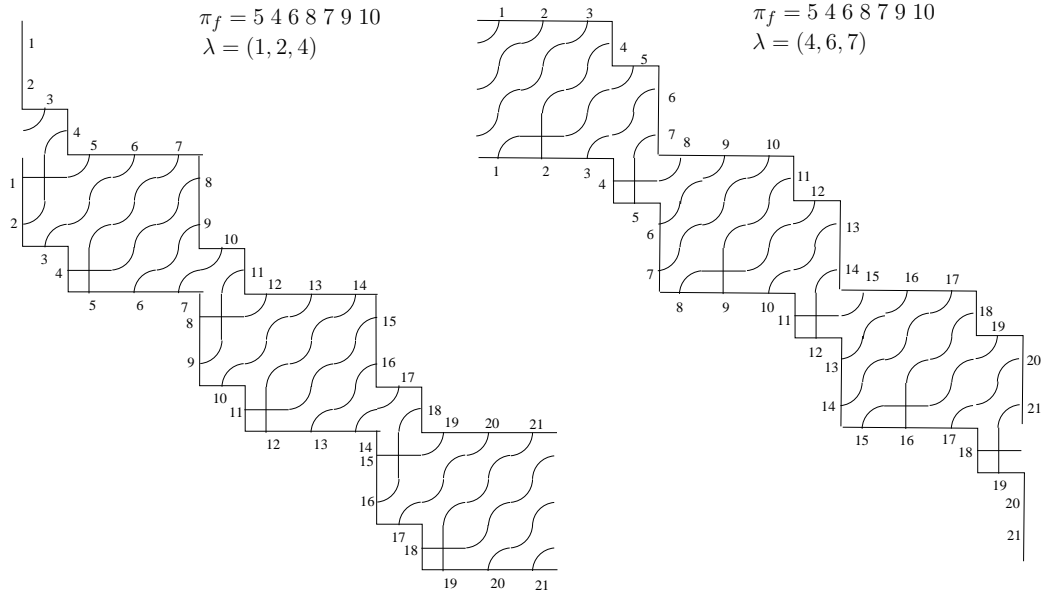}
\caption{For the affine permutation $\pi=5\;4\;6\;8\;7\;9\;10$, a representative affine pipe dream each for $\lambda=(1,2,4)$ and $\lambda=(4,6,7)$.}
\label{fig:twostrips}
\end{center}
\end{figure}

\pagebreak

\begin{Prop} If $P_\lambda$ is a (reduced) affine pipe dream constructed from $\lambda$ and $\pi$, then following the pipes from the lower line to the upper line gives the permutation of $\Z$ that takes $i \mapsto \pi(i)$. \end{Prop}

\begin{proof} The argument is the same as in the proof of Theorem 2.5 in \cite{KM03}: we use induction on the number of crosses. If we add a crossing tile to the antidiagonal labeled $i$, it switches the pipes starting at $i$ and $i+1$. As before, each inversion in $\pi$ contributes at least one crossing tile in $P_\lambda$, so the number of tiles is at least $\ell(\pi)$. If $P_\lambda$ is not reduced, an inversion may contribute more than one tile. Then the number of tiles equals $\ell(\pi)$ if and only if $P_\lambda$ is reduced.
\end{proof}

For a given permutation $\pi$ written as a word and the strip filled with simple transpositions as in \S \ref{sec:newtermorder}: starting with the identity permutation, we can read the word in order, and for each reflection $s_i$, if $s_i$ takes the word higher in Bruhat order but is still below $w$, then keep it. This will give us the lexicographically first reduced word, and putting $\textcross$'s in those boxes will give the top pipe dream. Similarly, the lexicographically last word will give us the bottom pipe dream. We also note that an $n$-throw in a pattern corresponds to a row of $\textcross$'s in the pipe dream, while a 0-throw gives a column of $\textcross$'s.

\subsection{Subword Complex, Stanley-Reisner Ring, and Gr{\"o}bner Basis}

Consider the strip filled with elbow $\textelbow$ tiles, to create pipes that go from the bottom line, northeast to the top line. Then our term order labeling is equivalent to labeling the antidiagonal containing $(i,j)$ with $s_k$ if replacing the elbows in $(i,j)$ with a crossing tile yields the transposition on $k$ and $k+1$. Note this matches our construction of $Q_\lambda$. We can now apply the following theorems:
\begin{Thm}\cite[Theorem 4]{Kfs}
  Let $f \in \Z[x_1,\ldots,x_n]$ be a degree $k$ polynomial
  whose lexicographically first term is a product of $k$ distinct variables.

  Let $Y$ be one of the schemes constructed from the hypersurface
  $f=0$ by taking components, intersecting, taking unions, and repeating.
  Then $Y$ is reduced over all but finitely many $p$, and over $\Q$.

  Let $\lambda$ is the lexicographic weighting
  $(\varepsilon,\varepsilon^2,\ldots,\varepsilon^n)$ on the variables.
  Let $\init Y$ be the initial scheme of $Y$. Then (away from those $p$)
  $\init Y$ is a Stanley-Reisner scheme.

\end{Thm}

 In particular, we have

\begin{Thm} \cite{Ksw}
$\text{lex}\init \af{f,\circ}^\lambda = SR(\Delta (Q_\lambda, f))$. \end{Thm}

With our main theorem, this implies that, as in the non-affine case, we have a subword complex and also that the transversal duals are the facets of a subword complex. Recall from \S \ref{sec:swc} that all subword complexes are vertex-decomposable.

\begin{Prop} The facets of $SR(\Delta(Q_\lambda,f))$ correspond to affine pipe dreams. \end{Prop}
\begin{proof}
As in the non-affine case, $SR(\Delta(Q_\lambda,f))$ is a subword complex. We define the correspondence as follows. A simple reflection $s_i$ is in a facet if and only if there is a corresponding $\textelbow$ tile in the pipe dream, meaning $s_i$ is not in the subword. Conversely, $s_i$ is not in any facet if and only if there is a corresponding $\textcross$ tile in the pipe dream, so $s_i$ is in the subword.
\end{proof}

Define the {\bf affine pipe dream complex} $\Delta(\pi, \lambda)$ to be the simplicial complex with vertices labeled by entries $(i,j)$ in the periodic strip and faces labeled by the elbow sets in the affine pipe dreams for $\pi$ with shape defined by $\lambda$.

\begin{Cor} The affine pipe dream complex for $f$ with $k$ balls is the subword complex $\Delta(Q,f)$.
\end{Cor}

\begin{cont_ex}{\bf\ref{ex:mainthmex}. } See Figure \ref{fig:affinepipedreamcomplex} for the complex $\Delta(Q_\lambda,f)$. Note that non-maximal faces may correspond to non-reduced affine pipe dreams.

\begin{figure}
\begin{center}
\scalebox{.5}{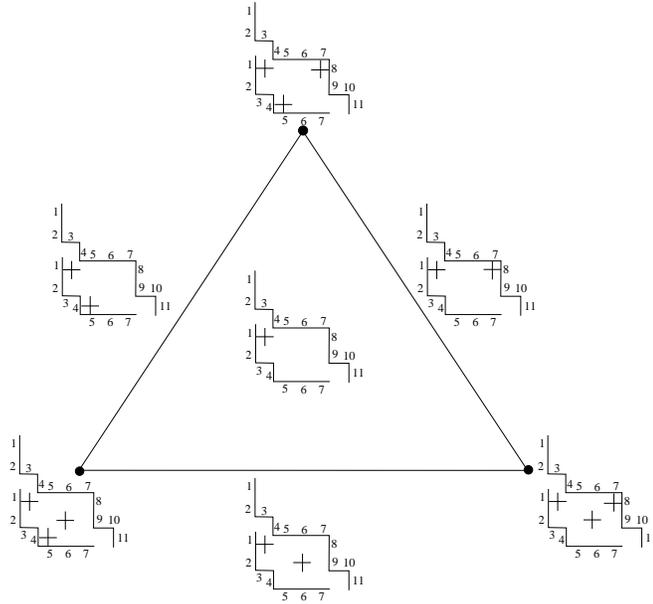}
\caption{The affine pipe dream complex for $\pi_\lambda=5\;4\;6\;8\;7\;9\;10$, with cone vertices removed.}
\label{fig:affinepipedreamcomplex}
\end{center}
\end{figure}

\end{cont_ex}

\begin{Prop} The rank conditions are a Gr{\"o}bner basis for the ideal $SR(\Delta(Q_\lambda,w))$. \end{Prop}
\begin{proof}  The new term order we defined in \S \ref{sec:newtermorder} picks out terms on a split antidiagonal, determined as follows: for the $k \times k$ submatrix starting with column $j$, find the cyclically last column $j'<j$ that is a column of the identity, with 1 in row $i'$. Then the antidiagonal in our submatrix picks out rows in the order $i', i'-1,\ldots,1, k, k-1, \ldots i'+1$. Since rank conditions are preserved under permutation of rows, we can permute the rows such that the term order picks out the diagonals. Then we can apply Theorem 1 in \cite{Sturm}.\\
By Theorem \ref{thm:kl}, we need only show this for the basic elements, since for a non-basic element  we can then construct the Gr{\"o}bner basis by concatenating the bases of the basic elements.\end{proof}

\subsection{The Ground State Case and Le/Cauchon Diagrams}

In this section, we will spell out our main isomorphism in the ground state case and relate it to \Le- and Cauchon diagrams, to show how our (more general) case is much richer.

For a partition $\lambda$, \cite{P} defines a {\bf \Le-diagram} (``Le" diagram) $D$ of shape $\lambda$ as
a filling of boxes of the Young diagram of shape $\lambda$ with $0$'s and $1$'s such that, for any three boxes
indexed $(i',j)$, $(i',j')$, $(i,j')$, where $i<i'$ and $j<j'$, filled with $a$, $b$, $c$, correspondingly,
if $a,c\neq 0$ then $b\neq 0$.  These three boxes should form the shape of a backwards ``L," pronounced ``el" (thus the name). See Figure \ref{fig:leex}. Let $\Le_{kn}$ be the set of \Le- diagrams whose shape $\lambda$ is contained in the $k\times (n-k)$ rectangle, and $|D|$ be the number of $1$'s in a diagram $D$.
\begin{figure}
\begin{center}
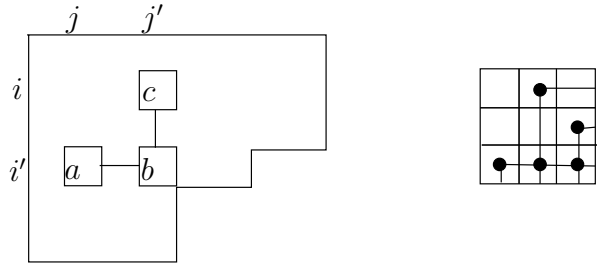
\caption{The Le diagram condition: $a,c\neq 0$ then $b\neq 0$, and an example of a Le diagram.}
\label{fig:leex}
\end{center}
\end{figure}
We will consider \Le- diagrams that are filled with 1's and 0's, where 1 denotes ``in" the diagram. For each 1, draw the {\bf hook}, a line going east and a line going south from the containing box. Then the \Le- condition is equivalent to requiring that there is a 1 at every intersection of hook lines. Let us call a 0 {\bf blocked} if there is a 1 above it in the same column. Then for each blocked 0, all entries to the west in the same row are also 0's.

 We now consider another type of diagram. As in \cite{GLL}, a {\bf Cauchon diagram} is an $m \times p$ grid of squares in which certain squares are black, according to the following rule: If a square is black, then either every square strictly to its left is black, or every square strictly above it is black. We let $C_{m,p}$ denote the set of $m \times p$ Cauchon diagrams, and say that a square indexed by $(i,j)$ {\bf belongs} to a diagram $C$ if it is black.

\begin{Lem} The \Le- diagrams and Cauchon diagrams are equivalent in the case where $\lambda$ is a rectangle. The bijection between Cauchon diagrams and \Le- diagrams maps black squares in a Cauchon diagram to boxes {\em not} in the \Le- diagram, and white squares in the Cauchon diagram to boxes {\em in} the \Le- diagram.\end{Lem}
\begin{proof}  A square is black in a Cauchon diagram if all the boxes to the west are black, or all the squares north are black. Then in the \Le- diagram, either there is no hook coming in from the west, or there is no hook coming in from the north. Then there is no hook crossing, so the box is not in the \Le- diagram. Note that satisfying the Cauchon condition that all west boxes are black or all above are black does not imply that a box is black.\end{proof}

 We define the subset of permutations called {\bf restricted permutations} $$S_{m+p}^{[-p,m]} := \{w \in S_{m+p} \; | \; -p \leq w(i)-i \leq m \text{ for all } i \in [1,m+p] \}.$$ Note that $S_{m+p}^{[-p,m]}  \subseteq S_{m+p}$, and $$S_{m+p}^{[-p,m]} = \{ w \in S_{m+p} \; | \; w \leq (m+1,m+2,\ldots,m+p,1,2,\ldots m)\}.$$

\begin{Thm} \cite[Theorem 24.1]{P} The nonempty totally nonnegative cells in $M_{mp}^{\geq 0}(\R)$ are indexed by $m \times p$ Cauchon diagrams. \end{Thm}
The $m \times p$ Cauchon diagrams biject with the restricted permutations $S_{m+p}^{[-p,m]}$, as shown in \cite{Hprime}, and there is a bijection between them using pipe dreams, shown in \S 19 of \cite{P}.

As in $\S 19$ of \cite{P}, we have a bijection between pairs of permutations $(u_D,w_\lambda)$ and \Le- diagrams $D$ of shape $\lambda$. Then we construct a wiring diagram as described in \S \ref{sec:perms}, now replacing each 1-box in $D$ with an elbow tile, and each 0-box with a crossing tile. Then we read off the permutation as before, denoting it $u_D \in S_n$.

\begin{Exa} See Figure \ref{fig:youngdia_2} for the construction of $(u_D,w_\lambda)=(31254,31452)$.\end{Exa}

\begin{figure}[htbp]
\begin{center}
\scalebox{0.6}{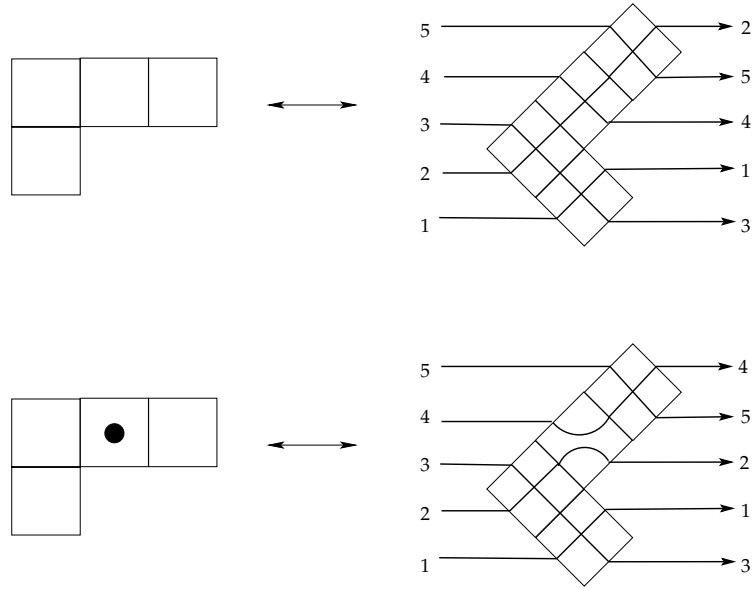}
\caption{Corresponding the permutations $u_D = 31254$ and $w_\lambda = 31452$ with the Le-diagram $D$ of shape $\lambda = (3,1)$.}
\label{fig:youngdia_2}
\end{center}
\end{figure}

\begin{Thm} [Theorem 19.1, \cite{P}]  $D \mapsto u_D$ is a bijection between $D$ of shape $\lambda$ and $u \in S_n$ such that $u \leq w_\lambda$ in Bruhat order, where $w_\lambda$ is the Grassmannian permutation associated to the Young diagram of shape $\lambda$. The number of 1's in $D$ is equal to $\ell(w_k)-\ell(u_D)$.\end{Thm}

We will consider affine pipe dreams in the ground state case, indexed by an arbitrary permutation $w$ and the patch is centered at $\lambda=(1,\ldots,k)$. The distinguished path gives an affine pipe dream, where a block in the strip is a $k \times (n-k)$ rectangle, and $v=w_0 w_0^P = (n-k+1,n-k+2,\ldots n, 1,2, \ldots, k)$, and $\Pi_w^v$  is the top point on $Gr_k \C^n$. See Figure \ref{fig:gspd} for the construction of an example of a ground state case.

\begin{figure}
\begin{center}
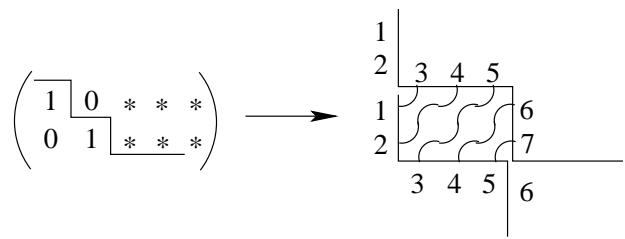
\caption{For $\lambda=(1,2)$ and $n=5$, this pipe dream corresponds to the pattern $22222$.}
\label{fig:gspd}
\end{center}
\end{figure}

\begin{Thm} Cauchon diagrams are in bijective correspondence with affine bottom pipe dreams. \end{Thm}
\begin{proof} Replace each black box in the Cauchon diagram $DS$ with a crossing tile, and each white box with an elbows tile. As in \S 19 of \cite{P}, the permutation associated with $D$ is read from the southwest corner, moving north and then east. Then flip $D$ vertically so that the pipes are read from the northwest corner, south then east, to match how pipes are read in pipe dreams. If an affine pipe dream $P$ is not a bottom pipe dream, then it has a possible chute or ladder move on box $(i,j)$. This implies that there is at least one box west of $(i,j)$ that is elbows, and at least one box south that is elbows. Then when we flip $P$ vertically to make a Cauchon diagram, that box (now black) has at least one white box west and at least one white box above, violating the Cauchon condition. \end{proof}

\begin{figure}
\begin{center}
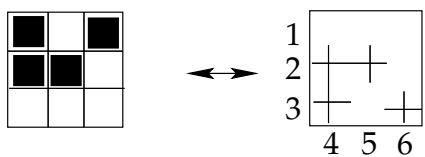
\caption{Cauchon diagram and affine bottom pipe dream for $w=143265$.}
\label{fig:cauchonex}
\end{center}
\end{figure}

See Figure \ref{fig:cauchonex} for an example of this correspondence. Thus, considering affine pipe dreams gives us a much more generalizable picture that correspond to Cauchon diagrams, and thus \Le- diagrams also, in the ground state case.

\end{document}